\documentclass[a4paper, 11pt]{article}

\usepackage{amsmath,amssymb,amsthm,centernot,indentfirst,mathrsfs, mathtools}
\usepackage{color}
\usepackage{ascmac}

\newtheorem{theorem}{Theorem}[section]
\newtheorem{lemma}[theorem]{Lemma}
\newtheorem{proposition}[theorem]{Proposition}

\newtheorem{corollary}[theorem]{Corollary}
\theoremstyle{definition}
\newtheorem{definition}[theorem]{Definition}

\theoremstyle{remark}
\newtheorem*{remark}{Remark}

\allowdisplaybreaks[4]
\numberwithin{equation}{section}

\newcommand{\R}{\mathbb{R}}

\newcommand{\Ent}{\mathrm{Ent}}

\newcommand{\K}{\mathcal{K}}

\begin{document}

\title{Analytic aspects of the dilation inequality for symmetric convex sets in Euclidean spaces}
\author{Hiroshi Tsuji\footnote{Department of Mathematics, Osaka University, Osaka 560-0043, Japan (tsuji@cr.math.sci.osaka-u.ac.jp), 
	2020 {\it Mathematics Subject Classification: 28A75, 26D10, 60E15,}  
	}}
\date{}
\maketitle

\begin{abstract}
We discuss an analytic form of the dilation inequality for symmetric convex sets in Euclidean spaces, which is a counterpart of analytic aspects of Cheeger's isoperimetric inequality. 
We show that the dilation inequality for symmetric convex sets is equivalent to a certain bound of the relative entropy for symmetric quasi-convex functions, which is close to the logarithmic Sobolev inequality or Cram\'{e}r--Rao inequality. 
As corollaries, we investigate the reverse Shannon inequality, logarithmic Sobolev inequality, Kahane--Khintchine inequality, deviation inequality and isoperimetry. 
We also give new probability measures satisfying the dilation inequality for symmetric convex sets via bounded perturbations and tensorization. 
\end{abstract}
\tableofcontents

\section{Introduction}

Cheeger's isoperimetric inequality with respect to a probability measure $\mu$ on $\R^n$ is one of the most important geometric inequalities in geometry and geometric analysis. 
Cheeger \cite{Ch70} and Maz'ya \cite{Ma62-1, Ma62-2} showed that Cheeger's isoperimetric inequality gives the spectral gap of the Laplace--Beltrami operator induced by $\mu$. 
Conversely, Buser \cite{Bu82} (see also Ledoux \cite{Le94}) also proved that the spectral gap, or equivalently the Poincar\'{e} inequality, gives Cheeger's isoperimetric inequality. 
Hence we can naturally regard the Poincar\'{e} inequality as an analytic form of Cheeger's isoperimetric inequality. 
Moreover, Bobkov--Houdr\'{e} \cite{BH97} also gave an equivalence between Cheeger's isoperimetric inequality and the $(1,1)$-Poincar\'{e} inequality, and thus the $(1,1)$-Poincar\'{e} inequality is another analytic aspect of Cheeger's isoperimetric inequality. 

On the other hand, Nazarov--Sodin--Volberg \cite{NSV02} showed a new sharp isoperimetric-type inequality for a log-concave probability measure on $\R^n$, which we call the dilation inequality in this paper. 
This inequality is originally given by Borell \cite{Bor74} and investigated by many researchers in \cite{LS93, Gu99, NSV02, B03, B07, BN07, Fr09, Kl17, Ts21} where the sharpness and generalization of the dilation inequality are discussed. 
Here a measure $\mu$ on $\R^n$ is log-concave if for any compact subsets $A, B \subset \R^n$, it holds 
$$
\mu((1-t) A + tB) \ge \mu(A)^{1-t} \mu(B)^t, \;\;\; \forall t \in (0, 1), 
$$
where $(1-t) A + tB \coloneqq \{ (1-t)a + tb \mid a \in A, b \in B\}$ is the Minkowski sum. 
For a Borel subset $A \subset \R^n$ and $\varepsilon \in (0,1)$, we define an $\varepsilon$-dilation of $A$ by 
\begin{equation}\label{e:Dil}
A_\varepsilon \coloneqq A \cup \left\{ x \in \R^n\, \Bigg{|}\, \exists y \in \R^n,  \int_0^1 \mathbf{1}_A((1-t)x+ty)\, dt > 1-\varepsilon \right\}
\end{equation}
and define a dilation area of $A$ by 
\begin{equation}\label{e:DilArea}
\mu^*(A) \coloneqq \liminf_{\varepsilon \downarrow 0} \frac{\mu(A_\varepsilon) - \mu(A) }{\varepsilon}. 
\end{equation}
An $\varepsilon$-dilation $A_\varepsilon$ is a counterpart of an $\varepsilon$-neighborhood $[A]_\varepsilon \coloneqq \{x \in \R^n \mid \exists a \in A, |x -a| < \varepsilon\}$, and a dilation area $\mu^*(A)$ is a counterpart of the $\mu$-perimeter of $A$ (or the Minkowski content of $A$ with respect to $\mu$) given by 
\begin{equation}\label{e:Area}
\mu^+(A) \coloneqq \liminf_{\varepsilon \downarrow 0} \frac{\mu([A]_\varepsilon) - \mu(A) }{\varepsilon}. 
\end{equation}
Then Nazarov--Sodin--Volberg \cite{NSV02} (see also \cite{Ts21}) showed that every log-concave probability measure $\mu$ on $\R^n$ satisfies 
\begin{equation}\label{e:Borell}
\mu^*(A) \ge - (1-\mu(A) ) \log (1-\mu(A))
\end{equation}
for any Borel subset $A \subset \R^n$.
We note that it is natural to consider \eqref{e:Borell} as a counterpart of Cheeger's isoperimetric inequality, namely  
\begin{equation}\label{e:ChIso}
\mu^+(A) \ge \kappa  \min\{ \mu(A), 1-\mu(A)\}
\end{equation}
for any Borel subset $A \subset \R^n$ with some $\kappa>0$. 
In fact, Kannan--Lov\'{a}sz--Simonovits \cite{KLS95} (see also \cite{B99, B07}) showed that every log-concave probability measure $\mu$ on $\R^n$ also satisfies \eqref{e:ChIso} with some positive constant depending on $\mu$. 
Moreover, on one hand, we see that the two-sided exponential measure $d\nu_2(x) = \frac12 e^{-|x|}\, dx$ on $\R$ satisfies $\nu_2^+((-\infty, x)) = \min\{\nu_2((-\infty, x), \nu_2((x, \infty))\}$ for any $x \in \R$, on the other hand, the one-sided exponential measure $d\nu_1(x) = e^{-x}\, dx$ on $(0, \infty)$ satisfies \eqref{e:Borell} with equality for $A=(0, x)$ for any $x>0$. 
Thus the both inequalities \eqref{e:Borell} and \eqref{e:ChIso} are sharp in the class of log-concave probability measures.

Our main goal in this paper is to investigate an analytic aspect of the dilation inequality \eqref{e:Borell} as the Poincar\'{e} inequality and $(1,1)$-Poincar\'{e} inequality are analytic forms of Cheeger's isoperimetric inequality. 
To this end, however the definition of the dilation \eqref{e:Dil} is complicated, and thus as the first step, we focus only on symmetric open convex sets $K \subset \R^n$ (we say that $K$ is symmetric if $K = -K$). 
In this case, it is known (see \cite{Fr09}) that an $\varepsilon$-dilation of $K$ can be represented simply as 
$$
K_\varepsilon = \frac{1+\varepsilon}{1-\varepsilon} K.
$$
From this property, we consider \eqref{e:Dil} as a generalization of the dilation. 
We also remark that on one hand, we have 
\begin{equation}\label{e:SpeDilK}
K_\varepsilon 
= K + \frac{2\varepsilon}{1-\varepsilon} K, 
\end{equation}
on the other hand, the $\varepsilon$-neighborhood of $K$ can be rewritten as 
$$
[K]_\varepsilon = K + \varepsilon {\rm B}_2^n, 
$$
where ${\rm B}_2^n \coloneqq \{ x \in \R^n \mid |x| < 1\}$ is the standard Euclidean open unit ball. 
Therefore the difference between the $\varepsilon$-neighborhood and $\varepsilon$-dilation is clear.

We also note that our restriction to symmetric open convex sets is enough to develop theory related to the dilation inequality. 
In fact, the dilation inequality is originally known as Borell's lemma such that 
$$
\mu( \R^n \setminus tK) \le \left( \frac{1-\mu(K)}{\mu(K)} \right)^{\frac{t+1}2} \mu(K), \;\;\; t \ge 1
$$
for any log-concave probability measure $\mu$ on $\R^n$ and symmetric convex set $K \subset \R^n$. 
In particular, it follows from Borell's lemma that 
\begin{equation}\label{e:ConcBorell}
\mu(\R^n \setminus tK) \le c e^{-Ct}, \;\;\; t \ge 1
\end{equation} 
whenever $\mu(K) \ge 2/3$, where $c, C>0$ are absolute constants. 
This inequality seems concentration of measure with respect to dilations, and we can observe the same inequality from \eqref{e:Borell} for log-concave probability measures (see \cite[Theorem 4.1]{Ts21}). 
By using \eqref{e:ConcBorell}, various geometric and analytic inequalities are induced like the Kahane--Khintchine inequality in \cite{Bor74, Gu99} (see also \cite{MS}) and Cheeger's isoperimetric inequality in \cite{B99}. 
We can see other applications of \eqref{e:ConcBorell} in \cite{BGVV14}.

To describe our results in this paper, we introduce some notions. 
Let $\Omega \subset \R^n$ be a symmetric convex domain, and let $\mathcal{K}_s^n(\Omega)$ be the set of all nonempty, symmetric open convex sets in $\R^n$. 

\begin{definition}\label{d:DilKappa}
A probability measure $\mu$ supported on a symmetric convex domain $\Omega \subset \R^n$ satisfies the dilation inequality for $\mathcal{K}_s^n(\Omega)$ with $\kappa>0$ if 
\begin{equation}\label{e:DilKappa}
\mu^*(K) \ge - \kappa (1-\mu(K)) \log (1-\mu(K)), \;\;\; \forall K \in \mathcal{K}_s^n(\Omega).
\end{equation}
\end{definition}

We may replace $\mathcal{K}_s^n(\Omega)$ by $\mathcal{K}_s^n(\R^n)$ in \eqref{e:DilKappa} since by the definition \eqref{e:Dil}, it holds that $(K \cap \Omega)_\varepsilon \subset K_\varepsilon$ for any $K \in \K_s^n(\R^n)$ and $\varepsilon \in (0,1)$, and thus $\mu^*(K \cap \Omega) \le \mu^*(K)$ by $\mu(K) = \mu(K \cap \Omega)$, associated with $K \cap \Omega \in \K_s^n(\Omega)$. 
We also remark that $\mu$ may not be symmetric even if its support is symmetric. 
As we have already mentioned, all log-concave probability measures on $\Omega$ (and thus on $\R^n$) satisfy the dilation inequality for $\K_s^n(\Omega)$ with $\kappa=1$. 
In particular, important examples are symmetric log-concave probability measures on $\R$ and the standard Gaussian measure $d\gamma_n \coloneqq (2\pi)^{- n/2}e^{-|x|^2/2}\, dx$ on $\R^n$. 
We can observe that these measures satisfy \eqref{e:DilKappa} with $\kappa=2$ (see Appendix). 

Next, we introduce the relative entropy. For a nonnegative Borel function $f$ and a probability measure $\mu$ on $\Omega$ with $\int_{\Omega}f\, d\mu<+ \infty$, we define the relative entropy of $f$ with respect to $\mu$ by 
$$
\Ent_\mu(f) \coloneqq \int_{\Omega} f \log f\, d\mu - \int_{\Omega} f\, d\mu \log \int_{\Omega} f\, d\mu, 
$$
where we put $0 \log 0 \coloneqq 0$. 
Jensen's inequality implies that the relative entropy is nonnegative, and is 0 if and only if $f$ is constant $\mu$-a.e., on $\Omega$.

The following functional inequalities, which are special cases we will show in Theorem \ref{t:FDI},   follow from \eqref{e:DilKappa}. 
\begin{theorem}\label{t:FDIIntro}
Let $\mu$ be a probability measure supported on a symmetric convex domain $\Omega$ and let $f \colon \Omega \to [0, \infty)$ be a continuous and symmetric function with $f \in L^1(\mu)$. 
We assume that $\mu$ satisfies the dilation inequality for $\K_s^n(\Omega)$ for some $\kappa>0$. 
\begin{itemize}
\item[(1)] If $f$ is convex on $\Omega$ with $\int_\Omega \inf_{y \in \partial f(x)} \langle x, y \rangle \, d\mu(x) < + \infty$, then it holds 
\begin{equation}\label{e:SpeFDI}
\Ent_\mu(f) \le \frac2\kappa \int_\Omega \inf_{y \in \partial f(x)} \langle x, y \rangle \, d\mu(x), 
\end{equation}
where $\partial f(x)$ is the subdifferential of $f$ at $x \in \Omega$. 
\item[(2)]
If $f$ is locally Lipschitz on $\{x \in \R^n \mid f(x) > f(0)\}$ and quasi-convex on $\Omega$, 
then it holds 
\begin{equation}\label{e:SpeFDI2}
\Ent_\mu(f) \le \frac2\kappa \int_{\{f>f(0)\}} |x| |\nabla f(x)| \, d\mu(x).
\end{equation}
In particular, when $f$ is locally Lipschitz and quasi-convex on $\Omega$, 
we have 
\begin{equation}\label{e:SpeFDI2.5}
\Ent_\mu(f) \le \frac2\kappa \int_{\Omega} |x| |\nabla f(x)| \, d\mu(x).
\end{equation}
In addition, when $f$ is $C^1$ on $\Omega$, 
then we have 
\begin{equation}\label{e:SpeFDI3}
\Ent_\mu(f) \le \frac2\kappa \int_\Omega \langle x,\nabla f(x) \rangle \, d\mu(x).
\end{equation}
\end{itemize} 
\end{theorem}

Here we say that a function $f \colon \Omega \to \R$ is quasi-convex if $\{ x \in \Omega \mid f(x) < \lambda \}$ is a convex set for any $\lambda \in \R$. In particular, quasi-convexity is a generalization of convexity. 
An important example is $|\cdot|^p$ for $p>0$, which is continuous and symmetric quasi-convex on $\R^n$ and locally Lipschitz on $\R^n \setminus \{0\}$. 
In addition, $|\cdot|^p$ is locally Lipschitz and convex on $\R^n$ when $p \ge 1$. 
See Section \ref{s:FDI} for more details and other examples of quasi-convex functions. 

We emphasize that Theorem \ref{t:FDIIntro} is the special case of Theorem \ref{t:FDI} where we will show a more general inequality for  functions in a more wider class. 
Moreover, we will actually confirm that Theorem \ref{t:FDI} can recover the dilation inequality \eqref{e:DilKappa} in Theorem \ref{t:ReConstDI}. 
In this sense, our theorem gives the optimal estimate. 

As the first application of Theorem \ref{t:FDIIntro}, we obtain the reverse Shannon inequality. 
\begin{corollary}\label{c:RevShannon}
Let $h \in C^1(\R^n)$ be a nonnegative function such that $h/\gamma_n$ is symmetric quasi-convex function with $\int_{\R^n} h(x)\, dx=1$ and 
\begin{equation}\label{e:TechAss}
\lim_{|x|\to +\infty} |x| h(x) = 0.
\end{equation}
Then it holds that 
\begin{equation}\label{e:RevShannon}
\int_{\R^n} h\log h\, dx 
\le
\frac12 \int_{\R^n} |x|^2h(x)\, dx - \frac n2 \log (2\pi e^2). 
\end{equation}
\end{corollary}
The classical Shannon inequality (for instance, see \cite{CT06}) implies the lower bound of the Shannon entropy such that 
$$
 \int_{\R^n} h \log h \,dx
\ge - \frac n2 \log \left( \frac{2\pi e}{n} \int_{\R^n} |x|^2 h(x)\, dx \right)
$$
for any nonnegative function $h$ on $\R^n$ with $\int_{\R^n} h\, dx=1$ and $\int_{\R^n} |x|^2h(x)\, dx<+\infty$. 
On the other hand, \eqref{e:RevShannon} gives the upper bound of the Shannon entropy. 
We remark that as we will see in Subsection \ref{s:Gauss}, we can ensure $\int_{\R^n} |x|^2h(x)\, dx \ge n$ in our settings, and thus it always holds that 
$$
- \frac n2 \log \left( \frac{2\pi e}{n} \int_{\R^n} |x|^2 h(x)\, dx \right)
\le
\frac12 \int_{\R^n} |x|^2h(x)\, dx - \frac n2 \log (2\pi e^2). 
$$
In addition, we can check that when $h=\gamma_n$, then equality in \eqref{e:RevShannon} holds. 

As another application of Theorem \ref{t:FDIIntro}, we can observe the logarithmic Sobolev type or Cram\'{e}r--Rao type inequality in the special case, which will be investigated in Subsection \ref{s:LSI}. 
In general, we say that a probability measure $\mu$ satisfies the logarithmic Sobolev inequality with $\rho>0$ if 
\begin{equation}\label{e:GrossLSI}
\Ent_\mu(f) \le \frac{1}{2\rho} I_\mu(f) 
\end{equation}
for any nonnegative locally Lipschitz function $f$ on $\R^n$, where $I_\mu(f)$ is the Fisher information of $f$ with respect to $\mu$ given by 
$$
I_\mu(f) \coloneqq \int_{\R^n} \frac{|\nabla f|^2}{f} \, d\mu.
$$
It is known that if $d\mu=e^{-\varphi}\, dx$ with $\varphi \in C^\infty(\R^n)$ satisfies $\nabla^2 \varphi \ge \rho$ for some $\rho>0$, then $\mu$ satisfies the logarithmic Sobolev inequality with  $\rho$. 
However, when $\nabla^2 \varphi \ge 0$ (which means that $\varphi$ is convex), $\mu$ may not satisfy \eqref{e:GrossLSI} for any $\rho>0$. 
Indeed, if $\mu$ satisfies the logarithmic Sobolev inequality, $\mu$ should satisfy the normal concentration, or equivalently $\int_{\R^n} e^{\varepsilon |x|^2}\,d\mu(x) < +\infty$ for some $\varepsilon>0$.  
In particular, since $\nabla^2 \varphi \ge 0$ is equivalent to the log-concavity of $\mu$ by \cite{Bor75}, we can observe that a log-concave probability measure may not satisfy \eqref{e:GrossLSI} for any $\rho>0$ in general. 
We refer a reader to \cite{BGL14} for details of the logarithmic Sobolev inequality. 
Nevertheless, we obtain the relation between the relative entropy and the Fisher information from Theorem \ref{t:FDIIntro} by the Cauchy--Schwarz inequality immediately. 

\begin{proposition}\label{p:LSIIntro}
Let $\mu$, $\Omega$ and $f$ be as in Theorem \ref{t:FDIIntro}. 
If $f$ is a locally Lipschitz and quasi-convex function on $\Omega$, 
then it holds 
\begin{equation}\label{e:SpeLSI}
\Ent_\mu(f) \le \frac{2}{\kappa} \left( \int_\Omega |x|^2 f(x)\, d\mu(x) \right)^{1/2} \sqrt{I_\mu(f)}. 
\end{equation}
\end{proposition}

We note that \eqref{e:SpeLSI} is also close to the Cram\'{e}r--Rao inequality. 
The classical Cram\'{e}r--Rao inequality (or the Heisenberg--Pauli--Weyl inequality) implies that 
for any nonnegative locally Lipschitz function $h$ with $\int_{\R^n} h\, dx=1$ and $\int_{\R^n} |x|^2h(x)\, dx<+\infty$, it holds 
\begin{equation}\label{e:CR}
n \le \left( \int_{\R^n} |x|^2 h(x)\, dx \right)^\frac12  \sqrt{I_{dx}(h)}.
\end{equation}
Our result \eqref{e:SpeLSI} does not induce \eqref{e:CR} since the relative entropy can take the value 0, and in this sense \eqref{e:SpeLSI} is different from the uncertainty principle. 
However, this difference is natural since on one hand, we cannot take any constant function in \eqref{e:CR}, on the other hand, we can take one in \eqref{e:SpeLSI} due to the finite mass of $\mu$. 
Nevertheless, the behavior of the relative entropy is closely related to the dimension. 
In fact, given probability measures $\mu_1, \mu_2$ and nonnegative functions $f_1, f_2$ on $\R^{n_1}$ and $\R^{n_2}$ with $\int_{\R^{n_1}}f_1(x)\, dx = \int_{\R^{n_2}}f_2(x)\, dx=1$, we can check that 
$$
\Ent_{\mu_1\otimes \mu_2} (f_{12})
=\Ent_{\mu_1}(f_1) + \Ent_{\mu_2}(f_2), 
$$
where $f_{12}(x_1, x_2) \coloneqq f_1(x_1)f_2(x_2)$ for $(x_1, x_2) \in \R^{n_1} \times \R^{n_2}$.
This implies that the relative entropy can be linear increasing in the dimension $n$, and in this sense, the bound \eqref{e:SpeLSI} is similar to \eqref{e:CR}. 

%

As we will see in Subsection \ref{s:KK}, we will also discuss Kahane--Khintchine inequalities with positive and negative exponents for symmetric quasi-convex functions via Theorem \ref{t:FDIIntro} (and Theorem \ref{t:FDI} and Proposition \ref{p:CoareaForm}), and discuss deviation inequalities as their application. 
Similar inequalities for general functions  have been already investigated in \cite{NSV02, B07, Fr09, Ts21} where we need to assume the Remez type inequality.  
On the other hand, we can obtain Kahane--Khintchine inequalities and deviation inequalities without the Remez type inequality. 
We enumerate our results only on deviation inequalities in special cases. 

\begin{corollary}\label{c:LDIntro}
Let $\mu$ and $\Omega$ be as in Theorem \ref{t:FDIIntro}. 
\begin{itemize}
\item[(1)] Let $f$ be a positive $C^1$ symmetric quasi-convex function on $\Omega$ satisfying 
$$
f \in \bigcap_{p \ge 1} L^p(\mu). 
$$ 
We set 
$$
\alpha \coloneqq \frac{\kappa}{ 2\| \langle \cdot, \nabla \log f( \cdot) \rangle \|_{L^\infty}}.
$$ 
If $1\le \alpha < +\infty$, then it holds that 
\begin{equation}\label{e:ConcIntro}
\mu( \{ x \in \Omega \mid f(x) \ge Ct \alpha^{-1/\alpha}\| f \|_{L^\alpha(\mu)} \}) \le 2 \exp (-t^\alpha), \;\;\; \forall t \ge 1, 
\end{equation}
where $C>0$ is an absolute constant. 
\item[(2)]
Suppose that $\Omega$ is bounded, and let $f$ be a positive $C^1$ symmetric quasi-convex function on some neighborhood of $\overline{\Omega}$ and set 
$$
\beta \coloneqq \frac{2}{\kappa \log 2} \| \langle \cdot, \nabla \log f( \cdot) \rangle \|_{L^\infty}.  
$$
Suppose that $0 < \beta < +\infty$ with $f^{-1/\beta} \in L^1(\mu)$. 
Then for any small enough $\varepsilon>0$, it holds that 
$$
\mu(\{ x \in \Omega \mid f(x) \le t\, {\rm med}(f) \}) \le \left(\frac{e}{\varepsilon \beta} \right)^{1-\varepsilon \beta} t^{\frac1\beta - \varepsilon}, \;\;\; \forall t \in (0, 1], 
$$
where ${\rm med}(f) \in \R$ is the L\'{e}vy mean of $f$, which means that 
$$
\mu(\{ x \in \Omega \mid f(x) \ge {\rm med}(f) \}) \ge \frac12, \;\;\;
\mu(\{ x \in \Omega \mid f(x) \le {\rm med}(f) \}) \ge \frac12.
$$
\end{itemize}
\end{corollary}

As the final application, we can also obtain the following result on isoperimetry. 

\begin{corollary}\label{c:IsoIntro}
Let $\mu=e^{-\varphi(x)}\, dx$ be a probability measure supported on a symmetric convex domain $\Omega \subset \R^n$. 
Suppose that $\varphi$ is smooth on some neighborhood of $\overline{\Omega}$ and $\mu$ satisfies the dilation inequality for $\K_s^n(\Omega)$ with $\kappa>0$. 
Then for any bounded $K \in \K_s^n(\Omega)$ with smooth boundary and $p \in (1,2]$, we have 
\begin{equation}\label{e:IsoIntro}
\mu^+(K)
\ge 
\left( \frac{r(K) }{\int_{\partial K} \langle x, \eta(x) \rangle |x|^{p'} e^{-\varphi(x)}\, d\sigma_K(x) } \right)^{p-1}
 \left[ -\frac\kappa2 (1-\mu(K)) \log (1-\mu(K)) \right]^p, 
\end{equation}
where $p'$ is the conjugate of $p$, $r(K)$ is the maximal constant $c>0$ such that $c {\rm B}_2^n \subset K$, $\eta$ is the outer unit normal vector along $\partial K$ and $\sigma_K$ is the surface measure on $\partial K$. 
\end{corollary}

Corollary \ref{c:IsoIntro} reminds us of Cheeger's isoperimetric inequality for log-concave probability measures by Kannan--Lov\'{a}sz--Simonovits \cite{KLS95} and Bobkov \cite{B99, Bo07} where the first or second moment appears as isoperimetric constants. 
Kannan--Lov\'{a}sz--Simonovits also conjecture that the isoperimetric constant of every log-concave probability measure is controlled by the covariance matrix, which is called the KLS conjecture. 
We refer a reader to \cite{BGVV14} for its history and related works and to \cite{Kl23} for the recent development. 

We remark that the dilation inequality \eqref{e:DilKappa} can give an estimate of the $\mu$-perimeter directly. 
Indeed, if $R(K)>0$ is the minimal constant $C>0$ such that $K \subset C {\rm B}_2^n$ for $K \in \K_s^n(\R^n)$, then it follows from \eqref{e:SpeDilK} that 
$$
K_\varepsilon 
\subset K + \frac{2\varepsilon}{1-\varepsilon} R(K) {\rm B}_2^n = [K]_{\frac{2\varepsilon}{1-\varepsilon} R(K)},  
$$
which implies that 
$$
\mu^*(K) \le 2R(K) \mu^+(K). 
$$
Combining this inequality with \eqref{e:DilKappa}, we conclude
\begin{equation}\label{e:IsoDirect}
\mu^+(K) \ge -\frac{\kappa}{2R(K)} (1-\mu(K) ) \log (1-\mu(K)).
\end{equation}
We can find a similar estimate in \cite{B99} for log-concave probability measures. 
However, \eqref{e:IsoIntro} seems different from  \eqref{e:IsoDirect} since \eqref{e:IsoIntro} requires not only the geometric structure of $K$, but also the  distribution. 
In particular, we can recover \eqref{e:IsoDirect} from \eqref{e:IsoIntro}. 
In fact, by the definition of $R(K)$, we have $|x| \le R(K)$ for $x \in \partial K$, which implies that  
$$
\int_{\partial K} \langle x, \eta(x) \rangle |x|^{p'} e^{-\varphi(x)}\, d \sigma_K(x)
\le
R(K)^{p'+1} \int_{\partial K} e^{-\varphi}\, d \sigma_K. 
$$
Hence \eqref{e:IsoIntro} yields that 
$$
\mu^+(K)
\ge 
\frac{1}{R(K)^{2p-1}} \left( \frac{r(K) }{\int_{\partial K} e^{-\varphi}\, d \sigma_K} \right)^{p-1}
 \left[ -\frac\kappa2 (1-\mu(K)) \log (1-\mu(K)) \right]^p, 
$$ 
and thus letting $p \downarrow 1$, we obtain \eqref{e:IsoDirect}.

In Section \ref{s:Re}, we will show the equivalence between \eqref{e:DilKappa} and Theorem \ref{t:FDI} which generalizes Theorem \ref{t:FDIIntro}, and as its corollaries, we will give new classes satisfying the dilation inequality \eqref{e:DilKappa}. 
More precisely, we will discuss the stability under bounded perturbations and tensor products.

\begin{corollary}\label{c:BDDIntro}
Let $\mu$ be a probability measure supported on a symmetric convex domain $\Omega \subset \R^n$ with $\int_\Omega |x| \, d\mu(x) < + \infty$ and let $h$ be a positive Borel function on $\Omega$ such that $b^{-1} \le h \le b$ for some $b>1$ and $\int_\Omega h \, d\mu=1$. 
Let $\nu$ be a probability measure on $\Omega$ given by $d\nu = h\, d\mu$. 
If $\mu$ satisfies the dilation inequality for $\K_s^n(\Omega)$ with $\kappa>0$, then $\nu$ satisfies the dilation inequality for $\K_s^n(\Omega)$ with the constant $b^{-2}\kappa$. 
\end{corollary}

\begin{corollary}\label{c:TensorProdIntro}
Let $\mu_1, \mu_2$ be probability measures supported on symmetric convex domains $\Omega_1 \subset \R$ and $\Omega_2 \subset \R^n$, respectively, with $\int_{\Omega_1} |x|\, d\mu_1, \int_{\Omega_2} |x|\, d\mu_2<+\infty$. 
We suppose that $\mu_1, \mu_2$ satisfy the dilation inequality for $\K_s^1(\Omega_1), \K_s^n(\Omega_2)$ with some $\kappa_1, \kappa_2>0$, respectively. 
Let $K \subset \R \times \R^n$ be an open convex set such that if $(x, y) \in K \subset \R \times \R^n$, then it holds that $(-x, y), (x, -y), (-x, -y) \in K$. 
Then $\mu_1 \otimes \mu_2$ satisfies \eqref{e:DilKappa} for $K$ with the constant $\kappa=(\kappa_1^{-1}+ \kappa_2^{-1})^{-1}$. 
\end{corollary}

The structure of the rest of this paper is as follows.
In Section \ref{s:FDI}, we introduce the class of functions including good enough symmetric quasi-convex functions and define certain derivative as a counterpart of the gradient.
After that, we show the functional form of the dilation inequality which leads to Theorem \ref{t:FDIIntro}. 
In Section \ref{s:App}, we give some applications which follow from Theorems \ref{t:FDIIntro} and \ref{t:FDI}. 
More precisely, we show the reverse Shannon inequality, logarithmic Sobolev inequality, Kahane--Khintchine inequality, deviation inequality and the estimate of isoperimetry. 
In the final section, we show the dilation inequality from the functional inequality constructed in Section \ref{s:FDI}, and confirm the equivalence between the dilation inequality and the functional form. 
As corollaries, we give stability results of the dilation inequality via bounded perturbations and tensorization.

\section{Functional inequality derived from the dilation inequality}\label{s:FDI}

Our goal in this section is to give a functional form of the dilation inequality \eqref{e:DilKappa} and show Theorem \ref{t:FDIIntro} as its special case. 

In what follows, let $\Omega \subset \R^n$ be a symmetric convex domain. 
We say that a function $f\colon \Omega \to \R$ is quasi-convex if a set 
$\{x \in \Omega \mid f(x) < \lambda\}$ is convex for any $\lambda \in \R$, or equivalently it holds that 
$$
f((1-t)x + ty) \le \max \{ f(x), f(y)\}, \;\;\; \forall x, y \in \Omega, \forall t \in [0,1]. 
$$ 
For instance, all convex functions are quasi-convex. 
Another example is $|\cdot|^p$ on $\R^n$ for $p>0$ which is quasi-convex, but not convex when $p \in (0,1)$. 
This example also implies that quasi-convexity does not yield convexity. 
It is also known that a continuous function $f$ on $\R$ is quasi-convex if and only if 
$f$ is either monotone on $\R$ or there exists some point $x_0 \in \R$ such that $f$ is monotone decreasing on $(-\infty, x_0]$ and monotone increasing on $[x_0, \infty)$ (see \cite[Proposition 3.8 and Proposition 3.9]{ADSZ10}). 
Moreover, if a function $f$ on $\Omega \subset \R^n$ is continuously differentiable, then 
quasi-convexity of $f$ is characterized by 
$$
\langle x - y, \nabla f(x) \rangle \ge 0
$$
for any $x, y \in \Omega$ with $f(x) \ge f(y)$ (see \cite[Theorem 3.1]{ADSZ10}). 
In particular, when $f$ is symmetric (which means that $f(x)=f(-x)$ for any $x \in \Omega$), then we have 
$$
\langle x, \nabla f(x) \rangle \ge 0, \;\;\; \forall x \in \Omega
$$
since $f(0) = \min_{x \in \Omega} f(x)$ by quasi-convexity and symmetry of $f$. 
A reader is referred to \cite{ADSZ10} for more information on quasi-convexity. 

Given a symmetric quasi-convex function $f \colon \Omega \to [0, \infty)$, we define a function $\Phi_f \colon \R^n \to [0, \infty]$ by 
\begin{equation}\label{e:DefPhi_f}
\Phi_f(x) \coloneqq \limsup_{\varepsilon \downarrow 0} \frac{f(x) - f( \frac{1-\varepsilon}{1+\varepsilon} x)}{\varepsilon}, \quad x \in \Omega.
\end{equation}
Since $f$ is a nonnegative and symmetric quasi-convex function, for any $\varepsilon \in (0,1)$ and $x \in \Omega$, it holds that $f(x) \ge f(\frac{1-\varepsilon}{1+\varepsilon} x)$, and thus that $\Phi_f$ is always nonnegative. 
In particular, when $f$ is continuously differentiable, we see that 
\begin{equation}\label{e:SpePhi_f1}
\Phi_f(x) =2 \langle x, \nabla f(x) \rangle, \;\;\; x \in \Omega.
\end{equation}
An important example is a norm. Let $\| \cdot \|_K$ for $K \in \K_s^n(\R^n)$ be a nonnegative function on $\R^n$ defined by 
\begin{equation}\label{e:Norm}
\|x\|_K \coloneqq \inf \{ \lambda >0 \mid x \in \lambda K \}, \;\;\;x \in \R^n. 
\end{equation}
We call it the gauge function of $K$. 
If $\overline{K}$ is a convex body, then the gauge function $\|\cdot\|_K$ is exactly a norm whose closed unit ball is $\overline{K}$. 
By the definition, we can immediately check that 
\begin{equation}\label{e:SpePhi_f2}
\Phi_{\| \cdot \|_K} = 2 \| \cdot \|_K
\end{equation}
since $\|\cdot\|_K$ is 1-homogeneous. 
We remark that $\| \cdot\|_K$ is not continuously differentiable at least at the origin. An advantage of the definition of \eqref{e:DefPhi_f} is that we may not suppose certain regularities of $f$ and thus we can consider a non-differentiable function on the whole space like a norm. 
We also note that by the definition, $\Phi_f$ is Borel measurable when $\Phi_f$ is finite and $f$ is continuous on $\Omega$.

Next, let $\mu$ be a probability measure supported on $\Omega$. 
We denote by $\mathrm{QC}(\Omega, \mu)$ all nonnegative, continuous and symmetric quasi-convex functions $f$ on $\Omega$ such that there exists a nonnegative Borel function $g \colon \Omega \to [0, \infty)$ in $L^1(\mu)$ and small enough $\varepsilon_0 \in (0, 1]$ satisfying 
\begin{equation}\label{e:CondQC}
\sup_{\varepsilon \in (0, \varepsilon_0)} \frac{f(x) - f( \frac{1-\varepsilon}{1+\varepsilon} x)}{\varepsilon} \le g(x), \;\;\; \forall x \in \Omega. 
\end{equation}
We may replace $\varepsilon_0$ with $1$ in \eqref{e:CondQC} when $f \in L^1(\mu)$. 
To see this, let $f,g$ and $\varepsilon_0$ be as above. 
Then we can observe that 
for any $\varepsilon \in (0, 1)$, 
$$
\frac{f(x) - f( \frac{1-\varepsilon}{1+\varepsilon} x)}{\varepsilon} 
\le g(x) + \frac1\varepsilon_0 f(x), \;\;\; \forall x \in \Omega.  
$$
This fact implies that we can take a function $\widetilde{g} \in L^1(\mu)$ satisfying 
$$
\sup_{\varepsilon \in (0, 1)} \frac{f(x) - f( \frac{1-\varepsilon}{1+\varepsilon} x)}{\varepsilon} \le \widetilde{g}(x), \;\;\; \forall x \in \Omega. 
$$

We remark that by the definition, if $f \in {\rm QC}(\Omega, \mu)$, then $af$ and $f + \alpha$ for any $a>0$ and $\alpha \ge - \inf_{x \in \Omega} f(x)$ also belong to ${\rm QC}(\Omega, \mu)$, and in particular, we have $\Phi_{af} = a\Phi_f$ and $\Phi_{f + \alpha} = \Phi_f$. 

An important example belonging to ${\rm QC}(\Omega, \mu)$ is a norm. 
Indeed, we can easily check that the gauge function $\| \cdot \|_K$ for $K \in \K_s^n(\Omega)$ is in ${\rm QC}(\Omega, \mu)$ when $\mu$ has finite first moment, namely $\int_\Omega |x|\, d\mu(x) < +\infty$. 
More generally, we can ensure that ${\rm QC}(\Omega, \mu)$ includes good locally Lipschitz and symmetric quasi-convex functions. 
Here a function $f \colon \Omega \to \R$ is locally Lipschitz if for any $x \in \Omega$, there exists some $r>0$ such that $f$ is Lipschitz on $B(x;r) \coloneqq \{ y \in \Omega \mid |x-y| <r\}$, or equivalently 
$$
| \nabla f(z) | \coloneqq \limsup_{y \to x} \frac{|f(y)-f(z)|}{|y-z|}
$$
is finite on $B(x;r)$. 

\begin{proposition}\label{p:ClassQC}
Let $\mu$ be a probability measure supported on a bounded symmetric convex domain $\Omega \subset \R^n$. 
Let $f$ be a nonnegative, continuous and symmetric quasi-convex function on some neighborhood of $\overline{\Omega}$. 
If $f$ is locally Lipschitz on $\overline{\Omega}$, then it holds that $f \in {\rm QC}(\Omega, \mu)$ and $\Phi_f(x) \le 2 |x| |\nabla f(x)|$ for any $x \in \Omega$. 
\end{proposition}

\begin{proof}
Since $f$ is locally Lipschitz, for any $x \in \overline{\Omega}$, there exist some $\varepsilon(x) \in (0, 1)$, $r(x)>0$ and $M(x)>0$ such that 
\begin{equation}\label{e:4/18}
\frac{|f(\frac{1-\varepsilon}{1+\varepsilon}y)-f(y)|}{|\frac{1-\varepsilon}{1+\varepsilon}y-y|} \le M(x), \;\;\; \forall \varepsilon \in (0, \varepsilon(x)), \forall y \in B(x, r(x)).
\end{equation}
Since $\overline{\Omega}$ is compact, we can take finite points $\{x_k\}_{k=1}^N \subset \overline{\Omega}$ ($N \in \mathbb{N}$) such that 
$$
\frac{|f(\frac{1-\varepsilon}{1+\varepsilon}y)-f(y)|}{|\frac{1-\varepsilon}{1+\varepsilon}y-y|} \le \max_{k=1,2,\dots, N}  M(x_k), \;\;\; \forall \varepsilon \in (0, \overline{\varepsilon}), \forall y \in \Omega, 
$$
where we set $\overline{\varepsilon} \coloneqq \min_{k=1, 2,\dots, N} \varepsilon(x_k)>0$. 
In particular, we obtain 
$$
\frac{|f(\frac{1-\varepsilon}{1+\varepsilon}y)-f(y)|}{\varepsilon}
\le
\frac{2}{1+\varepsilon} |y| \max_{k=1,2,\dots, N}  M(x_k)
\le
2{\rm diam}\, \Omega \max_{k=1,2,\dots, N}  M(x_k)
$$
for any $\varepsilon \in (0, \overline{\varepsilon})$ and $y \in \Omega$, which ensures \eqref{e:CondQC}. 
Hence we enjoy $f \in {\rm QC}(\Omega, \mu)$. 

In particular, by the definition, for any $\delta>0$ and $x \in \Omega$, we can take some $\varepsilon_0 \in (0,1)$ depending on  $\delta$ and $x$ such that 
$$
\frac{|f(\frac{1-\varepsilon}{1+\varepsilon}x)-f(x)|}{|\frac{1-\varepsilon}{1+\varepsilon}x-x|} \le
(1+\delta) |\nabla f(x)|, \;\;\; \forall \varepsilon \in (0, \varepsilon_0),  
$$
from which we see that 
\begin{align*}
\frac{f(x) - f( \frac{1-\varepsilon}{1+\varepsilon} x)}{\varepsilon}
\le
\frac{2(1+\delta)}{1+\varepsilon} |x||\nabla f(x)|
\le
2(1+\delta) |x||\nabla f(x)|. 
\end{align*}
Letting $\varepsilon \downarrow 0$, we have 
$$
\Phi_f(x)
\le
2(1+\delta) |x||\nabla f(x)|. 
$$
Since $\delta>0$ is arbitrary, we obtain $\Phi_f(x) \le 2 |x||\nabla f(x)|$ for any $x \in \Omega$.
\end{proof}

Moreover, when $f$ is a convex function instead of a quasi-convex function in Proposition \ref{p:ClassQC}, we can specify $\Phi_f$. 

\begin{proposition}\label{p:ClassQC2}
Let $\mu$ be a probability measure supported on a symmetric convex domain $\Omega \subset \R^n$. 
If a symmetric function $f \colon \Omega \to [0, \infty)$ is convex, then  it holds that $\Phi_f(x) = 2 \inf_{y \in \partial f(x)} \langle x, y \rangle$ for any $x \in \Omega$. 
Moreover, if we have $\int_\Omega \inf_{y \in \partial f(x)} \langle x, y \rangle \, d\mu(x) < + \infty$, then $f \in {\rm QC}(\Omega, \mu)$. 
Here $\partial f(x) \subset \R^n$ is the subdifferential of $f$ at $x \in \Omega$, namely
$$
y \in \partial f(x) 
\;\;\; \Longleftrightarrow \;\;\; 
\left[ f(z) \ge f(x) + \langle y, z-x \rangle, \;\;\; \forall z \in \Omega \right].
$$
\end{proposition}
We note that when $f$ is convex, the subdifferential of $f$ is always nonempty on $\Omega$, and in particular $\partial f(x) = \{ \nabla f(x)\}$ if $f$ is differentiable at $x \in \Omega$ (see \cite{Vi03}). 

\begin{remark}
The first assertion in Proposition \ref{p:ClassQC2} 
implies that $\inf_{y \in \partial f(x)} \langle x, y \rangle$ is Borel measurable since $\Phi_f$ is Borel measurable. 
\end{remark}

\begin{proof}
We fix $x \in \Omega$, and firstly show $\Phi_f(x) = 2 \inf_{y \in \partial f(x)} \langle x, y \rangle$. 
%
By the definition of the subdifferential of $f$, we have 
$$
f\left(\frac{1-\varepsilon}{1+\varepsilon} x \right) \ge f(x) + \left\langle y, \frac{1-\varepsilon}{1+\varepsilon} x-x \right\rangle
$$
for any $x \in \Omega$, $\varepsilon \in (0, 1)$ and $y \in \partial f(x)$, which is equivalent to 
$$
\frac{f(x) - f(\frac{1-\varepsilon}{1+\varepsilon} x )}{\varepsilon}
\le
\frac{2}{1+\varepsilon} \langle x, y\rangle. 
$$
Hence we obtain 
\begin{equation}\label{e:Cal5/5-2}
\frac{f(x) - f(\frac{1-\varepsilon}{1+\varepsilon} x )}{\varepsilon}
\le
2 \inf_{y \in \partial f(x)} \langle x, y\rangle. 
\end{equation}
Letting $\varepsilon \downarrow 0$, we conclude $\Phi_f(x) \le 2 \inf_{y \in \partial f(x)} \langle x, y \rangle$.

Next we show $\Phi_f(x) \ge 2 \inf_{y \in \partial f(x)} \langle x, y \rangle$. 
Let $\{\varepsilon_k\}_{k \in \mathbb{N}} \subset (0,1)$ be a monotone decreasing sequence satisfying $\lim_{k \to +\infty} \varepsilon_k=0$ and 
$$
\lim_{k \to +\infty} \frac{f(x)-f(x_k)}{\varepsilon_k} = \Phi_f(x), 
$$
where we set $x_k \coloneqq \frac{1-\varepsilon_k}{1+\varepsilon_k}x$, and take $y_k \in \partial f(x_k)$ for each $k \in \mathbb{N}$. 
If we can take a subsequence $\{ y_{k_\ell}\}_{\ell \in \mathbb{N}}$ of $\{ y_k\}_{k \in \mathbb{N}}$ such that  $y_{k_\ell}=0$ for all $\ell \in \mathbb{N}$, then for any $z \in \Omega$ and $\ell \in \mathbb{N}$, we have $f(z) \ge f(x_{k_\ell})$. 
Letting $\ell\to +\infty$, we enjoy $f(z) \ge f(x)$ for any $z \in \Omega$, which yields $0 \in \partial f(x)$. 
Hence we have $\inf_{y \in \partial f(x)} \langle x, y \rangle=0$. 
On the other hand, we also see that $f(x_{k_\ell})=f(0)$ since $f$ is symmetric and convex. 
Thus it follows that $f(tx)=f(0)$ for any $t \in [0, 1]$, which implies $\Phi_f(x)=0$. 
Hence we have $\Phi_f(x) = 0=2 \inf_{y \in \partial f(x)} \langle x, y \rangle$. 

Therefore we may suppose that $y_k \neq 0$ for all $k \in \mathbb{N}$. 
In addition, without loss of generality, we may suppose that $\{x_k\}_{k \in \mathbb{N}} \subset B(x; r/2)$ and $\overline{B(x;r)}  \subset \Omega$ for some $r>0$, where $B(x;r) \coloneqq \{ w \in \R^n \mid |x-w|<r\}$. 
By the definition of the subdifferential, it holds that 
\begin{equation}\label{e:Cal5/5}
f(z) \ge f(x_k) + \langle y_k, z-x_k \rangle, \;\;\; \forall z \in \Omega
\end{equation}
for each $k \in \mathbb{N}$. 
Inserting $z_k \coloneqq \frac{ry_k}{2|y_k|}+x_k$ in $z$, we obtain 
$$
f(z_k) \ge f(x_k) + \frac{r}{2} |y_k|. 
$$
Here we remark that $|x -z_k| \le r/2 + |x_k-x| <r$, and thus $z_k \in \Omega$. 
Moreover we have $z_k, x_k \in B(x;r)$. Hence since $\overline{B(x;r)} \subset \Omega$ and $f$ is continuous, we have 
$$
|y_k| \le \frac1r \max_{w \in  \overline{B(x;r)}} f(w), \;\;\; \forall k \in \mathbb{N}. 
$$
Hence we can take a subsequence of $\{y_k\}_{k \in \mathbb{N}}$ converging to some $\widetilde{y} \in \R^n$. 
Without loss of generality, we may suppose that $\lim_{k \to +\infty} y_k =\widetilde{y} $. 
Letting $k \to +\infty$ in \eqref{e:Cal5/5}, we have 
$$
f(z) \ge f(x) + \langle \widetilde{y} , z-x \rangle, \;\;\; \forall z \in \Omega, 
$$
which implies that $\widetilde{y}  \in \partial f(x)$. 
Moreover, it follows from \eqref{e:Cal5/5} that 
$$
f(x) \ge f(x_k) + \langle y_k, x-x_k \rangle, 
$$
which yields that 
$$
\frac{f(x) - f(x_k)}{\varepsilon_k}
\ge
\frac{2}{1+\varepsilon_k} \langle y_k, x \rangle. 
$$
Letting $k \to +\infty$, we obtain that 
$$
\Phi_f(x)  \ge 2 \langle \widetilde{y} , x \rangle  
\ge 
2 \inf_{y \in \partial f(x)} \langle x, y \rangle.
$$
This is the desired assertion. 

Finally, 
 \eqref{e:Cal5/5-2} and $\int_\Omega \inf_{y \in \partial f(x)} \langle x, y \rangle \, d\mu(x) < + \infty$ imply that $f \in {\rm QC}(\Omega, \mu)$.  
\end{proof}

To show our main result, we firstly give the following co-area formula associated with the dilation area, which has been appeared in \cite{Ts21} with a more weaker form. 

\begin{proposition}\label{p:CoareaForm}
Let $\mu$ be a probability measure supported on a symmetric convex domain $\Omega \subset \R^n$ and let $p>0$. 
Then for any nonnegative function $f$ with $f^p \in \mathrm{QC}(\Omega, \mu)$, we have 
\begin{equation}\label{e:CoareaForm}
\int_0^{\infty} t^{p-1}\mu^*(\{ x \in \R^n \mid f(x) < t\} ) \,dt \le \int_\Omega f^{p-1}\Phi_{f} \,d\mu.
\end{equation}
Moreover, for any positive function $f$ with $f^p \in \mathrm{QC}(\Omega, \mu)$, we also have 
\begin{equation}\label{e:NegCoareaForm}
\int_0^{\infty} t^{p-1} \mu^*\left(\left\{ x \in \R^n \,\Bigg{|}\, \frac{1}{f(x)} > t \right\} \right) \,dt \le  \int_\Omega f^{-p-1} \Phi_{f} \,d\mu.
\end{equation}
\end{proposition}

We remark that $\mu^*(\{ x \in \R^n \mid f(x) < t\} )$ is Borel measurable in $t$ since $\{ x \in \R^n \mid f(x) < t\}$ is monotone in $t$ and $\mu$ has the finite mass. 

\begin{proof}
Since $f^p \in \mathrm{QC}(\Omega, \mu)$ and $p>0$, $f$ is a nonnegative, continuous and symmetric quasi-convex function, and thus $\{ x \in \R^n \mid f(x) < \lambda\}$ is a symmetric open convex set for any $\lambda > 0$, from which it holds 
$$
\{ x \in \R^n \mid f(x) < \lambda\}_\varepsilon 
= \frac{1+ \varepsilon}{1-\varepsilon} \{ x \in \R^n \mid  f(x) < \lambda \}
=  \left\{ x \in \R^n \Bigg{|}  f \left(\frac{1- \varepsilon}{1+\varepsilon}x \right) < \lambda \right\}
$$ 
for any $\lambda > 0$.
Hence it follows from Fatou's lemma and $\mu(\R^n)=1$ that 
\begin{align*}
&\int_0^{\infty} t^{p-1} \mu^*(\{x \in \R^n \mid f(x) < t\}) \, dt 
\\
=&
\int_0^{\infty}t^{p-1} \liminf_{\varepsilon \downarrow 0}  \frac{\mu(\{ x\in \R^n \mid f(x) < t \}_{\varepsilon}) -\mu(\{ x \in \R^n \mid f(x) < t \})}{\varepsilon}\, dt 
\\
\le& 
\liminf_{\varepsilon \downarrow 0} \int_0^{\infty}  t^{p-1} \frac{\mu(\{ x \in \R^n \mid f(x) < t \}_{\varepsilon}) -\mu(\{ x \in \R^n \mid f(x) < t \})}{\varepsilon}\, dt 
\\
=&
\liminf_{\varepsilon \downarrow 0} \int_0^{\infty} \frac{1}{\varepsilon} t^{p-1}\left[ \mu\left( \left\{ x \in \R^n \Bigg{|}  f \left(\frac{1- \varepsilon}{1+\varepsilon}x \right) < t \right\} \right) -\mu( \{ x \in \R^n \mid f(x) < t \} ) \right]\, dt 
\\
=& 
\liminf_{\varepsilon \downarrow  0} \int_0^{\infty} \frac{1}{\varepsilon} t^{p-1}\left[\mu(\{ x \in \R^n \mid f(x) \ge t\}) - \mu\left( \left\{ x \in \R^n \Bigg{|}  f \left(\frac{1- \varepsilon}{1+\varepsilon}x \right) \ge t \right\} \right) \right]\, dt 
\\
=& 
\frac1p \liminf_{\varepsilon \downarrow 0} \int_{\Omega} \frac{1}{\varepsilon} \left(f^p(x) - f^p\left(\frac{1-\varepsilon}{1+\varepsilon} x\right) \right) \, d\mu(x).
\end{align*}
Moreover, by \eqref{e:CondQC} and $f^p \in {\rm QC}(\Omega, \mu)$, we can justify 
\begin{align*}
\liminf_{\varepsilon \downarrow 0} \int_{\Omega} \frac{1}{\varepsilon} \left(f^p(x) - f^p\left(\frac{1-\varepsilon}{1+\varepsilon} x\right) \right) \, d\mu(x)
\le&
\limsup_{\varepsilon \downarrow 0} \int_{\Omega} \frac{1}{\varepsilon} \left(f^p(x) - f^p\left(\frac{1-\varepsilon}{1+\varepsilon} x\right) \right) \, d\mu(x)
\\
\le&
\int_\Omega \Phi_{f^p}(x)\, d\mu(x), 
\end{align*}
where we used Fatou's lemma again. 
Hence we obtain 
$$
\int_0^{\infty} t^{p-1} \mu^*(\{x \in \R^n \mid f(x) < t\}) \, dt 
\le
\frac1p \int_\Omega \Phi_{f^p}(x)\, d\mu(x). 
$$
Since we see that $\Phi_{f^p} = p f^{p-1} \Phi_f$ by the definition of $\Phi_f$ and continuity of $f$, we can conclude \eqref{e:CoareaForm}. 

Next, we show \eqref{e:NegCoareaForm}. 
We remark that $a \coloneqq \inf_{x \in \Omega} f(x) = f(0) >0$ since $f>0$ on $\Omega$ and $f$ is a symmetric quasi-convex function by $p>0$.
Moreover, $\{ x \in \R^n \mid f(x)^{-1} > t \}$ is a symmetric open convex set for any $t>0$ since $f$ is a continuous  and symmetric quasi-convex function. 
As in the above argument, Fatou's lemma yields that 
\begin{align*}
&\int_0^{\infty} t^{p-1} \mu^*\left(\left\{ x \in \R^n \,\Bigg{|}\, \frac{1}{f(x)} > t \right\} \right) \, dt 
\\
\le& 
\liminf_{\varepsilon \downarrow 0} \int_0^{\infty} \frac1\varepsilon t^{p-1} \left[ \mu \left(\left\{ x \in \R^n \,\Bigg{|}\, \frac{1}{f(\frac{1-\varepsilon}{1+\varepsilon}x)} > t \right\} \right) -\mu \left( \left\{ x \in \R^n \,\Bigg{|}\, \frac{1}{f(x)} > t \right\} \right) \right]\, dt 
\\
=&
\frac1p \liminf_{\varepsilon \downarrow  0} \int_\Omega \frac{1}{\varepsilon} \left( \frac{1}{f^p \left(\frac{1- \varepsilon}{1+\varepsilon}x \right)}  - \frac{1}{f^p(x)} \right)\, d\mu(x).
\end{align*}
Since we see that 
\begin{align*}
\frac{1}{\varepsilon} \left( \frac{1}{f^p \left(\frac{1- \varepsilon}{1+\varepsilon}x \right)}  - \frac{1}{f^p(x)} \right)
\le
a^{-2p} \frac{1}{\varepsilon} \left( f^p(x) - f^p \left(\frac{1- \varepsilon}{1+\varepsilon}x \right) \right)
\end{align*}
and since $f^p \in \mathrm{QC}(\Omega, \mu)$, we can apply Fatou's lemma to see that 
\begin{align*}
\liminf_{\varepsilon \downarrow  0} \int_\Omega \frac{1}{\varepsilon} \left( \frac{1}{f^p \left(\frac{1- \varepsilon}{1+\varepsilon}x \right)}  - \frac{1}{f^p(x)} \right)\, d\mu(x)
\le&
\limsup_{\varepsilon \downarrow  0} \int_\Omega \frac{1}{\varepsilon} \left( \frac{1}{f^p \left(\frac{1- \varepsilon}{1+\varepsilon}x \right)}  - \frac{1}{f^p(x)} \right)\, d\mu(x)
\\
\le&
\int_\Omega \frac{1}{f(x)^{2p}} \Phi_{f^p}(x)\, d\mu(x). 
\end{align*}
Finally using $\Phi_{f^p} = p f^{p-1} \Phi_f$, we obtain \eqref{e:NegCoareaForm}. 


\end{proof}

Let ${\rm QC}^p(\Omega, \mu)$ for $p >0$ be the set of all functions $f$ on $\Omega$ such that $f^p \in {\rm QC}(\Omega, \mu) \cap L^1(\mu)$. 
Our main theorem in this section is the following. 

\begin{theorem}\label{t:FDI}
Let $\mu$ be a probability measure supported on a symmetric convex domain $\Omega \subset \R^n$. 
We assume that $\mu$ satisfies the dilation inequality for $\K_s^n(\Omega)$ for some $\kappa>0$. 
Then for any $f \in \mathrm{QC}^1(\Omega, \mu)$, it holds 
\begin{equation}\label{e:FDI}
\Ent_\mu(f) \le \frac1\kappa \int_\Omega \Phi_f \, d\mu.
\end{equation}
In addition, when $f \in C^1(\Omega)$, we obtain \eqref{e:SpeFDI3}. 
\end{theorem}

Our proof of this claim is almost same as the proof of \cite[Theorem 5.3]{Ts21}. 
For the completeness, we give the proof of Theorem \ref{t:FDI} here.

\begin{proof}
Since $f \in {\rm QC}(\Omega, \mu)$, we can apply \eqref{e:CoareaForm} with $p=1$ for sublevel sets of $f$. 
It follows from \eqref{e:DilKappa}, $\mu(\Omega)=1$ and \eqref{e:CoareaForm} that we have 
\begin{align}\label{e:m1}
- \int_0^{\infty} \mu(A_f(t)) \log\mu(A_f(t)) \,dt \le \frac{1}{\kappa}\int_{\Omega} \Phi_f\, d\mu, 
\end{align}
where we defined 
$$
A_f(t) \coloneqq \{ x \in \Omega \mid f(x) \ge t\}, \;\;\; t \ge 0.
$$
To see \eqref{e:FDI}, 
without loss of generality, we may assume that $\int_\Omega f\, d\mu=1$. In fact, 
we know $f \in L^1(\mu)$ and 
$$
\Phi_{af} = a \Phi_f, \;\;\; \Ent_\mu(af) = a \Ent_\mu(f)
$$
for any $a>0$ and $f \in {\rm QC}^1(\Omega, \mu)$ from which we can add the condition $\int_\Omega f\, d\mu=1$. 
Now, recall the dual formula of the relative entropy: for any continuous function $h\colon\R^n \to [0, \infty)$ with $\int_\Omega h \,d\mu =1$, 
it holds that 
\begin{equation}\label{e:DualEnt}
\Ent_\mu(h) =\sup_{\varphi \in C_b(\Omega)}\left[ \int_{\Omega} h\varphi \,d\mu - \log\int_{\Omega} e^{\varphi} \,d\mu\right], 
\end{equation}
where $C_b(\Omega)$ is the set of all bounded continuous functions on $\Omega$ (for instance, see \cite{S, Go22} and their proofs). 
Hence, since $\Ent_{\mu}(\mu(A)^{-1}\mathbf{1}_A) = -\log\mu(A)$ and $\int_\Omega \mu(A)^{-1}\mathbf{1}_A \,d\mu=1$ for any Borel subset $A \subset \Omega$ with $\mu(A)>0$, 
it holds that 
\begin{align*}
&-\int_0^{\infty} \mu(A_{f}(t)) \log\mu(A_{f}(t)) \,dt
\\
=& 
\int_0^\infty \mu(A_f(t)) \Ent_\mu(\mu(A_f(t))^{-1}\mathbf{1}_{A_f(t)}) \,dt 
\\
=& 
\int_0^\infty \sup_{\varphi \in C_b(\Omega)}\left[ \int_\Omega \varphi \mathbf{1}_{A_f(t)}\, d\mu - \mu(A_f(t))\log\int_\Omega e^{\varphi} \,d\mu\right] \,dt
\\
\ge &
\sup_{\varphi \in C_b(\Omega)} \left[ \int_\Omega\int_0^\infty \varphi \mathbf{1}_{A_f(t)}\, dtd\mu - \int_0^\infty \mu(A_f(t))\, dt \log\int_\Omega e^{\varphi}\, d\mu\right] 
\\
=&
\sup_{\varphi \in C_b(\Omega)} \left[\int_\Omega \varphi f \,d\mu - \log\int_\Omega e^{\varphi} \, d\mu\right] 
\\
=&
\Ent_\mu(f), 
\end{align*}
where we used $\int_0^\infty \mathbf{1}_{A_f(t)}(x) \,dt=f(x)$ for every $x \in \Omega$ and $\int_0^\infty \mu(A_f(t))\, dt = \int_\Omega f\, d\mu =1$. 
Combining this with \eqref{e:m1},  we conclude the desired assertion. 
\end{proof}

We conclude this section by giving the proof of Theorem \ref{t:FDIIntro}. 

\begin{proof}[Proof of Theorem \ref{t:FDIIntro}]
(1) immediately follows from Theorem \ref{t:FDI} and Proposition \ref{p:ClassQC2}. 

Let us show (2). 
Take an increasing sequence $\{\Omega_k\}_{k\in \mathbb{N}}$ such that $\Omega_k$ is an open, bounded symmetric convex set with $\overline{\Omega_k} \subset \Omega$ and $\lim_{k \to +\infty} \Omega_k = \Omega$. 
We can take such a sequence, for instance, by considering $\{ x \in \Omega \mid \|x\|_\Omega < 1-1/k\} \cap (k {\rm B}_2^n)$. 
Let $\mu_k$ be a normalized probability measure of $\mu$ on $\Omega_k$, namely $d\mu_k \coloneqq \mu(\Omega_k)^{-1}\mathbf{1}_{\Omega_k}\, d\mu$.  
Next, let $f$ be a function given in Theorem \ref{t:FDIIntro} (2), and for $\ell, m \in \mathbb{N}$, we set $f_{\ell, m}(x) \coloneqq \max\{\min\{f(x), \ell\}, f(0) +1/m \}$ for $x \in \Omega$.
Then since $f$ is locally Lipschitz on $\{x \in \Omega \mid f(x)>f(0)\}$, $f_{\ell,m}$ is locally Lipschitz on $\Omega$ for any $\ell,m \in \mathbb{N}$. 
In addition, $f_{\ell,m}$ is symmetric quasi-convex with $f_{\ell, m} \in L^1(\mu_k)$ for any $k,\ell,m \in \mathbb{N}$ by its construction. 
Hence applying Proposition \ref{p:ClassQC}, we have $f_{\ell, m} \in {\rm QC}^1 (\Omega_k, \mu_k)$ for any $k,\ell,m \in \mathbb{N}$. 
Moreover, $\mu_k$ satisfies the dilation inequality for $\K_s^n(\Omega_k)$ with $\kappa$. 
To see this, let $K \in \K_s^n(\Omega_k)$. 
Then by the definition, we see that 
$$
\mu_k^*(K) = \frac{1}{\mu(\Omega_k)}\mu^*(K). 
$$
Hence \eqref{e:DilKappa} yields that 
\begin{align*}
\mu_k^*(K) \ge& -\frac{\kappa}{\mu(\Omega_k)}(1-\mu(K)) \log (1-\mu(K))
\\
=&
-\frac{\kappa}{\mu(\Omega_k)}(1-\mu(\Omega_k) \mu_k(K)) \log (1-\mu(\Omega_k) \mu_k(K)). 
\end{align*}
By the elementary inequality $- \theta^{-1}(1-\theta x)\log (1-\theta x) \ge - (1-x)\log (1-x)$ for $\theta \in (0,1)$ and $x \in (0,1)$, we can obtain 
$$
\mu_k^*(K)
\ge
-\kappa(1- \mu_k(K)) \log (1-\mu_k(K)), 
$$ 
which is the desired assertion. 

Thus we can apply Theorem \ref{t:FDI} to see that 
$$
\Ent_{\mu_k}(f_{\ell,m}) \le \frac1\kappa \int_{\Omega_k} \Phi_{f_{\ell,m}} \, d\mu_k.
$$
Since $f_{\ell, m}$ is a bounded continuous function, the lower semi-continuity of the relative entropy (which follows from \eqref{e:DualEnt}) and the monotone convergence theorem as $k \to +\infty$ imply that 
$$
\Ent_\mu(f_{\ell,m}) \le \frac1\kappa \int_\Omega \Phi_{f_{\ell,m}} \, d\mu.
$$
As we mentioned in Proposition \ref{p:ClassQC}, we have 
$$
\Phi_{f_{\ell,m}}(x)
\le
\begin{cases}
2|x||\nabla f(x)| & \text{if $f(0) + 1/m < f(x) \le \ell$} \\
0  &\text{otherwise}
\end{cases}
$$
for $x \in \Omega$. 
Moreover, we see that 
$$
\begin{cases}
2|x||\nabla f(x)| & \text{if $f(0) + 1/m < f(x) \le \ell$} \\
0  &\text{otherwise}
\end{cases}
\le
\begin{cases}
2|x||\nabla f(x)| & \text{if $f(x) >f(0)$} \\
0  &\text{otherwise}
\end{cases}
$$
for $x \in \Omega$. 
In particular, we can replace the right hand side above by $2|x||\nabla f(x)|$ for $x \in \Omega$ if $f$ is locally Lipschitz on $\Omega$ since $|x||\nabla f(x)| \ge 0$ for $x \in \Omega$.  
Hence we have 
$$
\Ent_\mu(f_{\ell,m}) \le \frac1\kappa \int_{\{f>f(0)\}} |x||\nabla f(x)| \, d\mu(x), 
$$
and in particular, when $f$ is locally Lipschitz on $\Omega$, then we have 
$$
\Ent_\mu(f_{\ell,m}) \le \frac1\kappa \int_\Omega |x||\nabla f(x)| \, d\mu(x). 
$$
Thus by $\lim_{\ell,m \to +\infty} f_{\ell,m}=f$ and the lower semi-continuity of the relative entropy, we obtain \eqref{e:SpeFDI2} and \eqref{e:SpeFDI2.5}. 

When $f$ is $C^1$, we enjoy 
$$
\Phi_{f_{\ell, m}}(x) = \langle x, \nabla f(x) \rangle \mathbf{1}_{\{f(0)+1/m < f \le \ell\}}(x) \le \langle x, \nabla f(x) \rangle, \;\;\; x\in \Omega, 
$$
where we used $\langle x, \nabla f(x) \rangle \ge 0$ for any $x \in \Omega$ since $f$ is a symmetric quasi-convex function.
Using this formula directly instead of $2|x||\nabla f(x)|$, we can also obtain \eqref{e:SpeFDI3}. 
\end{proof}

%

\section{Some applications of Theorem \ref{t:FDI}}\label{s:App}

\subsection{Comparisons of the relative entropy, Wasserstein distance and variance}\label{s:Gauss}

As the first application of Theorem \ref{t:FDI}, we give comparisons of the relative entropy, Wasserstein distance and the variance in the case of the Gaussian measure. 
We denote the standard Gaussian measure on $\R^n$ by $d\gamma_n = (2\pi)^{-n/2}e^{-\frac12|x|^2}\, dx$.
To see our results, we introduce the $L^2$-Wasserstein distance, which appears in optimal transport theory. 

Let $\mu, \nu$ be probability measures on $\R^n$ with finite second moment. 
Then the $L^2$-Wasserstein distance of $\mu$ and $\nu$ is given by 
$$
W_2(\mu, \nu) \coloneqq \left\{ \inf_{\pi \in \Pi(\mu, \nu)} \int_{\R^n \times \R^n} |x-y|^2\, d\pi(x, y)\right\}^{1/2}, 
$$
where $\Pi(\mu, \nu)$ is the set of all couplings $\pi$ between $\mu$ and $\nu$, namely $\pi$ is a probability measure on $\R^n \times \R^n$ such that 
$\pi(A \times \R^n)=\mu(A)$ and $\pi(\R^n \times A)=\nu(A)$ for any Borel subset $A \subset \R^n$. 
It is known that $W_2$ is a distance function on the set of all probability measures on $\R^n$ with finite second moment. We refer a reader to \cite{Vi03, Vi09} for optimal transport theory and its related topics.

\begin{proposition}\label{p:Gauss}
Let $f \in C^1(\R^n)$ be a nonnegative, symmetric quasi-convex function with $\int_{\R^n} f\, d\gamma_n=1$ and 
\begin{equation}\label{e:TechAss5/5}
\lim_{|x|\to +\infty} |x|f(x)\gamma_n(x) = 0. 
\end{equation}
Then we have 
\begin{equation}\label{e:EntVar}
 \Ent_{\gamma_n}(f) \le \int_{\R^n} |x|^2 f(x)\, d\gamma_n(x) - n 
\end{equation}
and 
\begin{equation}\label{e:WVar}
\frac12 W_2^2(\gamma_n, \nu) \le \int_{\R^n} |x|^2 f(x)\, d\gamma_n(x) - n ,
\end{equation}
where $d\nu \coloneqq f\, d\gamma_n$.
\end{proposition}

We remark that \eqref{e:EntVar} in particular implies that every $C^1$ symmetric quasi-convex function $f \colon \R^n \to [0, \infty)$ as in Proposition \ref{p:Gauss} satisfies
\begin{equation}\label{e:VarEst}
\int_{\R^n} |x|^2 f(x)\, d\gamma_n(x) \ge n, 
\end{equation}
and equality holds if and only if $f \equiv 1$ on $\Omega$. 
Moreover, when the deficit of \eqref{e:VarEst} is small such that  
$$
\int_{\R^n} |x|^2 f(x)\, d\gamma_n(x) - n \le \varepsilon
$$
for small enough $\epsilon>0$, then \eqref{e:EntVar} and \eqref{e:WVar} imply that $f$ is close to the constant 1 in the both senses of the relative entropy and the $L^2$-Wasserstein distance. 

We also note that we have the trivial upper bound of $W_2(\gamma_n, \nu)$ such as 
$$
\frac12 W_2^2(\gamma_n, \nu) \le \int_{\R^n} |x|^2 f(x)\, d\gamma_n(x) + n. 
$$
Thus \eqref{e:WVar} strengthen this trivial bound. 

\begin{proof}
We first note that the standard Gaussian measure satisfies \eqref{e:DilKappa} for $\K_s^n(\R^n)$ with $\kappa=2$ (see Appendix).  
Hence by Theorem \ref{t:FDI} (and Theorem \ref{t:FDIIntro}), we have 
$$
\Ent_{\gamma_n}(f) \le \int_{\R^n} \langle x, \nabla f(x) \rangle \, d\gamma_n(x).
$$
Moreover, by \eqref{e:TechAss5/5}, integrating by parts yields that 
$$
\int_{\R^n} \langle x, \nabla f(x) \rangle \, d\gamma_n(x)
=
\int_{\R^n} |x|^2 f(x)\, d\gamma_n(x) - n. 
$$
Combining these facts, we obtain \eqref{e:EntVar}. 

The second assertion \eqref{e:WVar} immediately follows from \eqref{e:EntVar} and Gaussian Talagrand's transportation inequality (see \cite{Vi03, Vi09}) 
$$
\frac{1}{2} W_2^2(\gamma_n, \nu) \le \Ent_{\gamma_n}(f). 
$$
\end{proof}

\begin{proof}[Proof of Corollary \ref{c:RevShannon}]
We set $f \coloneqq h/\gamma_n$, then we can check that $f$ satisfies the assumptions in Proposition \ref{p:Gauss}.   Hence applying Proposition \ref{p:Gauss} to $f$, we see that 
\begin{align*}
\int_{\R^n} h\log h\, dx 
\le& \int_{\R^n} h \log \gamma_n\, dx + \int_{\R^n}|x|^2h(x)\, dx - n
\\
=&
\frac12 \int_{\R^n} |x|^2h(x)\, dx - \frac n2 \log (2\pi e^2),  
\end{align*}
which is the desired assertion. 
\end{proof}

\subsection{Cram\'{e}r--Rao inequality and logarithmic Sobolev inequality}\label{s:LSI}

From Theorem \ref{t:FDI} (and Theorem \ref{t:FDIIntro}), we can obtain the following logarithmic Sobolev type inequality or Cram\'{e}r--Rao type inequality, which includes Proposition \ref{p:LSIIntro}. 

\begin{proposition}\label{p:LSI}
Let $\mu$ and $\Omega$ be as in Theorem \ref{t:FDI} and
let $f \in L^1(\mu)$ be a nonnegative, locally Lipschitz and symmetric quasi-convex function. 
Then it holds that 
\begin{equation}\label{e:LSI1}
\Ent_\mu(f) \le \frac{2}{\kappa} \left( \int_\Omega |x|^2 f(x)\, d\mu(x) \right)^{1/2} \sqrt{I_\mu(f)}
\end{equation}
and 
\begin{equation}\label{e:LSI2}
\Ent_\mu(f) \le \frac{1}{\kappa} I_\mu(f) + \frac{1}{\kappa} \int_\Omega |x|^2 f(x)\, d\mu(x).
\end{equation}
\end{proposition}

\begin{proof}
Let $f$ be a function satisfying our assumptions. 
Then by Theorem \ref{t:FDIIntro}, we see that 
$$
\Ent_\mu(f) \le \frac2\kappa \int_\Omega |x||\nabla f(x)|\, d\mu.
$$ 
Then \eqref{e:LSI1} follows by combining this with the Cauchy--Schwarz inequality, and \eqref{e:LSI2} follows from \eqref{e:LSI1} and the arithmetic-geometric mean inequality. 
\end{proof}

As we described in our introduction, \eqref{e:LSI1} is close to the logarithmic Sobolev inequality, and exactly gives 
$$
\Ent_\mu(f) \le \frac{2}{\kappa} \left( \int_\Omega |x|^2 f(x)\, d\mu(x) \right)^{1/2} I_\mu(f)
$$ 
if $I_\mu(f) \ge 1$, and 
$$
\Ent_\mu(f) \le \frac{4}{\kappa^2}  \int_\Omega |x|^2 f(x)\, d\mu(x)  I_\mu(f)
$$ 
if $\Ent_\mu(f) \ge 1$.  
We emphasize that the constant of \eqref{e:LSI1} is given only by $\kappa$ and $\int_\Omega |x|^2 f(x)\, d\mu(x)$. The logarithmic Sobolev inequality also appears in \cite{Ts21} where we need the Poincar\'{e} constant of $\mu$. 

On the other hand, \eqref{e:LSI2} is close  to the defective logarithmic Sobolev inequality.  
Here we say that a probability measure $\mu$ on $\R^n$ satisfies the defective logarithmic Sobolev inequality with constants $\rho>0$ and $\tau \ge 0$ if 
$$
\Ent_\mu(f) \le \frac{1}{2\rho} I_\mu(f) + \tau \int_\Omega f\, d\mu
$$
for any nonnegative locally Lipschitz function $f$ on $\Omega$. 
We refer a reader to \cite{BGL14} for details of the defective logarithmic Sobolev inequality. 
In our case, when $\Omega$ is bounded, since $\Omega$ is symmetric, \eqref{e:LSI2} implies that  
\begin{equation}\label{e:4/13}
\Ent_\mu(f) \le \frac{1}{\kappa} I_\mu(f) + \frac{1}{4\kappa} ({\rm diam}\,\Omega)^2 \int_\Omega f\, d\mu. 
\end{equation}
In particular, when $\Omega$ is an interval in $\R$, we also obtain the logarithmic Sobolev inequality associated with the Poincar\'{e} constant of $\mu$. Here we say that $\mu$ satisfies the Poincar\'{e} inequality with constant $C_\mu>0$ if 
$$
C_\mu \int_\Omega f^2\, d\mu \le \int_\Omega | \nabla f|^2\, d\mu
$$
for any locally Lipschitz function $f$ on $\Omega$ with $\int_\Omega f\, d\mu=0$. 

\begin{corollary}\label{c:1-dimLSI}
Let $\mu$ be a symmetric probability measure on a bounded symmetric open interval $I \subset \R$ and $f \colon I \to \R$ be a locally Lipschitz function. 
Suppose that $\mu$ satisfies the dilation inequality for $\K_s^1(I)$ with $\kappa>0$ and the Poincar\'{e} inequality with $C_\mu>0$. 
In addition, we suppose that $f$ is odd and monotone function with $f\in L^2(\mu)$. 
Then it holds that 
\begin{equation}\label{e:1-dimLSI}
\Ent_\mu(f^2) \le \frac{1}{\kappa} \left( 4 + \frac{1}{4 C_\mu} ({\rm diam}\,I)^2 \right)  \int_I |f'|^2\, d\mu.
\end{equation}
In particular, when $d\mu = e^{-\varphi(x)}\, dx$ is log-concave, then we have 
\begin{equation}\label{e:1-dimLSILC}
\Ent_\mu(f^2) \le \left( 2 + \frac{1}{8 e^{-2\varphi(0)}} ({\rm diam}\,I)^2 \right)  \int_I |f'|^2\, d\mu.  
\end{equation}
\end{corollary}

\begin{proof}
First we remark that $f^2$ is a nonnegative, locally Lipschitz and symmetric quasi-convex function. 
Indeed, since $|f|^2$ is decreasing on $I \cap (-\infty, 0]$ and increasing on $I \cap [0, \infty)$ by the monotonicity of $f$, $|f|^2$ is quasi-convex. 
Moreover $|f|^2$ is locally Lipschitz and symmetric since $f$ is locally Lipschitz and odd. 
Applying Proposition \ref{p:LSI}, in particular \eqref{e:4/13}, to $f^2$, we obtain 
$$
\Ent_\mu(f^2) \le \frac{4}{\kappa} \int_I |f'|^2\, d\mu + \frac{1}{4\kappa} ({\rm diam}\,I)^2 \int_I f^2\, d\mu. 
$$
On the other hand, since $f$ is odd and $\mu$ is symmetric from which we have $\int_I f \, d\mu=0$, we can apply the Poincar\'{e} inequality to $f$ to see that 
$$
C_\mu \int_I f^2\, d\mu \le \int_I |f'|^2\, d\mu.
$$
Therefore we enjoy 
$$
\Ent_\mu(f^2) \le \frac{4}{\kappa} \int_I |f'|^2\, d\mu + \frac{1}{4\kappa C_\mu} ({\rm diam}\,I)^2\int_I |f'|^2\, d\mu, 
$$
which implies \eqref{e:1-dimLSI}. 

For the second assertion, we employ the result by Bobkov \cite{B99} where it is shown that every log-concave probability measure $\mu=e^{-\varphi}\, dx$ on $\R$ satisfies Cheeger's isoperimetric inequality with the constant $2e^{-\varphi(m)}$, where $m \in I$ is the median of $\mu$. 
In our case, since $\mu$ is symmetric, we can take $m$ as $0$. 
Hence, Cheeger's inequality \cite{Ch70} (see also \cite[Theorem 14.1.6]{BGVV14}) implies that $C_\mu \ge e^{-2\varphi(0)}$. 
\eqref{e:1-dimLSILC} follows from combining  \eqref{e:1-dimLSI} with the bound of the Poincar\'{e} constant and $\kappa=2$. The latter follows from every symmetric log-concave probability measure satisfying \eqref{e:DilKappa} with $\kappa=2$ (see Appendix). 
\end{proof}

\subsection{Kahane--Khintchine inequalities and deviation inequalities}\label{s:KK}

In this subsection, we consider deviation inequalities as described in Corollary \ref{c:LDIntro}. 
To see this, we give the following moment estimate for positive exponent which is a generalization of the comparison result of moments for log-concave probability measures, firstly discussed by Borell \cite{Bor74} (see also \cite{MS, GNT14, Mi23}). 

\begin{proposition}\label{p:Moment}
Let $\mu$ and $\Omega$ be as in Theorem \ref{t:FDI} and $p_0 >1$. 
If a nonnegative function $f$ on $\Omega$ satisfies 
$$
f \in \bigcap_{1 \le p \le p_0} {\rm QC}^p(\Omega, \mu),
$$ 
it holds that 
\begin{equation}\label{e:Moment}
\| f \|_{L^q(\mu)} \le \left( \frac{q}{p} \right)^{\frac1\kappa \| f^{-1} \Phi_f \|_{L^\infty(\{f>0\})}} \| f \|_{L^p(\mu)}
\end{equation}
for any $1\le p \le q <p_0$, where 
$$
\| f^{-1} \Phi_f \|_{L^\infty(\{f>0\})} \coloneqq {\rm ess\,sup} \{ f^{-1}(x) \Phi_f(x) \mid x \in \Omega, f(x)>0\}.
$$ 
In particular, if $f$ is in $C^1(\Omega)$ and satisfies 
\begin{equation}\label{e:GoodAss}
f \in \bigcap_{p \ge 1} {\rm QC}^p(\Omega, \mu), 
\end{equation}
then we have 
\begin{equation}\label{e:MomentSm}
\| f \|_{L^q(\mu)} \le \left( \frac{q}{p} \right)^{\frac2\kappa \| \langle \cdot, \nabla \log f(\cdot) \rangle \|_{L^\infty(\{f>0\})}} \| f \|_{L^p(\mu)}
\end{equation}
for any $1\le p \le q$. 
\end{proposition}

\begin{proof} 
Set 
$$
\Lambda(t) \coloneqq \frac{1}{t} \log \int_{\Omega} f^t\, d\mu, 
$$
then we see that 
\begin{align*}
\Lambda'(t) =& - \frac{1}{t^2} \log \int_{\Omega} f^t\, d\mu + \frac{1}{t} \frac{ \int_{\Omega} f^t\log f\, d\mu}{\int_{\Omega} f^t\, d\mu}
\\
=&
\frac{1}{t^2} \frac{1}{ \int_{\Omega} f^t\, d\mu} \Ent_{\mu}(f^t). 
\end{align*}
Since $f \in {\rm QC}^t(\Omega, \mu)$ for any $1\le t < p_0$, it follows from Theorem \ref{t:FDI} that 
\begin{align*}
\Lambda'(t) 
\le& \frac{1}{t^2} \frac{1}{\kappa} \frac{1}{ \int_{\Omega} f^t\, d\mu} \int_{\Omega} \Phi_{f^t}\, d\mu
\\
=&
\frac{1}{t} \frac{1}{\kappa} \frac{1}{ \int_{\Omega} f^t\, d\mu} \int_{\{ x \in \Omega \mid f(x)>0\}} f^{t-1} \Phi_f\, d\mu
\\
\le& \frac{1}{\kappa t} \| f^{-1}\Phi_f \|_{L^\infty(\{f>0\})}. 
\end{align*}
Here we used the fact that $\Phi_{f^t}(x)=0$ for any $t>0$ if $f(x)=0$ at $x \in \Omega$, which follows from $f^t \in {\rm QC}(\Omega, \mu)$. 
Hence integrating the above inequality from $p$ to $q$ with $1 \le p \le q <p_0$ yields the desired assertion \eqref{e:Moment}. 

\eqref{e:MomentSm} also follows by the same proof above and by \eqref{e:SpePhi_f1}.  
\end{proof}

For instance, when $\mu$ is log-concave, the gauge function $\| \cdot \|_K$ for $K \in \K_s^n(\R^n)$ satisfies \eqref{e:GoodAss}, and we can check $\| (\|\cdot\|_K)^{-1} \Phi_{\|\cdot\|_K} \|_{L^\infty(\{\|\cdot\|_K>0\})} =2$. 
Hence, \eqref{e:Moment} yields 
$$
\| \|\cdot\|_K \|_{L^q(\mu)} \le \left( \frac{q}{p} \right)^2 \| \|\cdot\|_K \|_{L^p(\mu)}
$$
for any $1 \le p \le q$ since all log-concave probability measures satisfy the dilation inequality with $\kappa=1$. In particular, when $\mu$ is symmetric on $\R$, since we can take $\kappa=2$ as we see in Appendix, we also obtain 
$$
\| |\cdot| \|_{L^q(\mu)} \le \frac{q}{p} \| |\cdot| \|_{L^p(\mu)}
$$
for any $1 \le p \le q$. 
It is known that the order of $q/p$ above is optimal (for instance, see \cite{GNT14}). 
On the other hand, it is known that all log-concave probability measures on $\R^n$ satisfy 
$$
\| \|\cdot\|_K \|_{L^q(\mu)} \le C \frac{q}{p} \| \|\cdot\|_K \|_{L^p(\mu)}
$$
for any $1 \le p \le q$, where $C>0$ is an absolute constant (see \cite{MS, BGVV14, GNT14}). 
We remark that the similar inequality for general functions has already appeared in \cite{Ts21} (see also \cite{B07, Fr09}). 
More precisely, in \cite{Ts21}, we need the Remez function to construct the moment comparison like \eqref{e:Moment}. 
For $s\ge 1$, we define $u_f(s) \ge 1$ by the best constant $C\ge 1$ such that 
$$
\{ x \in \Omega \mid f(x) \le \lambda\}_{1-1/s} \subset \{x \in \Omega \mid f(x) \le \lambda u_f(s)\}, \;\;\; \forall \lambda >0.
$$
We call the function $u_f \colon [1, \infty) \to [1, \infty)$ the Remez function of $f$, and set 
$$
u_f'(1) \coloneqq \limsup_{s\downarrow 1} \frac{u(s)-1}{s-1}. 
$$
Then it follows from \cite[Corollary 5.7]{Ts21} that 
$$
\| f \|_{L^q(\mu)} \le \left( \frac{q}{p} \right)^{u_f'(1)} \| f \|_{L^p(\mu)}
$$
for any nonnegative integrable function $f$ with $u_f'(1)<+\infty$ and for any $1 \le p \le q$. 
We remark that $\| f^{-1} \Phi_f \|_{L^\infty(\{f>0\})} \le u_f'(1)$ holds when $f$ is a continuous and symmetric quasi-convex function. 
Indeed by the definition of $u_f$, we have 
$$
f(x) \le f\left(\frac{1-\varepsilon}{1+\varepsilon} x \right) u_f \left( \frac{1}{1-\varepsilon} \right), \;\;\; \forall \varepsilon (0, 1), \forall x \in \Omega. 
$$
Hence it holds that $\Phi_f(x) \le f(x)u_f'(1)$ for $x \in \Omega$, from which we obtain $\| f^{-1} \Phi_f \|_{L^\infty(\{f>0\})} \le u_f'(1)$.

As a corollary of Proposition \ref{p:Moment}, we give a tail estimate of a measure. 
To see this, we introduce the Orlicz norm $\| \cdot\|_{\psi_\alpha}$ for $\alpha \ge 1$. 
Given any $\alpha \ge 1$ and Borel function $f \colon \Omega \to \R$, we set 
$$
\| f \|_{\psi_\alpha}
\coloneqq 
\inf \left\{
t>0 \,\Bigg{|}\, \int_\Omega \exp \left( \left(\frac{|f(x)|}{t} \right)^\alpha \right)\, d\mu \le 2
\right\}. 
$$
It is known that the Orlicz norm $\| \cdot\|_{\psi_\alpha}$ is also given by $L^p$-norms for $p\ge \alpha$ (see \cite[Lemma 2.4.2]{BGVV14}). 

\begin{lemma}\label{l:Orlicz}
Let $\alpha \ge1$ and $f \colon \Omega \to \R$ be a Borel function. 
Then 
$$
\| f \|_{\psi_\alpha} \simeq \sup_{p \ge \alpha} \frac{\| f \|_{L^p(\mu)}}{p^{1/\alpha}}. 
$$
Here $A \simeq B$ means that there exist some absolute constants $c, C>0$ such that $cB \le A \le C B$. 
\end{lemma}

By Proposition \ref{p:Moment}, we obtain the following estimate of some Orlicz norm and the deviation inequality. 

\begin{corollary}\label{c:Conc}
Let $\mu$ and $\Omega$ be as in Theorem \ref{t:FDI} and let $f$ be a nonnegative function on $\Omega$ satisfying \eqref{e:GoodAss}. 
We set 
$$
\alpha \coloneqq \frac{\kappa}{ \| f^{-1} \Phi_f\|_{L^\infty(\{ f>0\})}}.
$$ 
If $1 \le \alpha < +\infty$, then it holds that 
\begin{equation}\label{e:Orlicz}
\| f \|_{\psi_\alpha} \simeq \alpha^{-\frac{1}{\alpha}} \| f \|_{L^\alpha(\mu)}. 
\end{equation}
In addition, we have 
\begin{equation}\label{e:Conc}
\mu( \{ x \in \Omega \mid f(x) \ge Ct \alpha^{-1/\alpha}\| f \|_{L^\alpha(\mu)} \}) \le 2 \exp (-t^\alpha), \;\;\; \forall t \ge 1, 
\end{equation}
where $C>0$ is an absolute constant. 
\end{corollary}

\begin{proof}
\eqref{e:Orlicz} is a direct consequence of Proposition \ref{p:Moment} and Lemma \ref{l:Orlicz}. 
\eqref{e:Conc} also follows from \eqref{e:Orlicz} and Markov's inequality. 
\end{proof}

\begin{proof}[Proof of Corollary \ref{c:LDIntro} (1)]
Let $f$ be a function given in Corollary \ref{c:LDIntro} (1). 
Then for any $t \ge 1$, $f^t$ is a positive $C^1$ symmetric quasi-convex function on $\Omega$ with $f^t \in L^1(\mu)$. 
Hence by Theorem \ref{t:FDIIntro}, we have 
$$
\Ent_\mu(f^t) \le \frac{2t}{\kappa} \int_\Omega f^{t-1} \langle x, \nabla f(x) \rangle \, d\mu(x). 
$$
Applying this inequality instead of Theorem \ref{t:FDI} in the proof of Proposition \ref{p:Moment}, we obtain 
$$
\| f \|_{L^q(\mu)} \le \left( \frac{q}{p} \right)^{\frac2\kappa \| \langle x, \nabla \log f \rangle \|_{L^\infty}} \| f \|_{L^p(\mu)}
$$
for any $1 \le p \le q$. 
Finally by the same  argument as in the proof of Corollary \ref{c:Conc}, we conclude the desired assertion. 
\end{proof}

Next, we consider the Kahane--Khintchine inequality for negative exponent via Proposition \ref{p:CoareaForm}. 

\begin{proposition}\label{p:NegMom}
Let $\mu$ and $\Omega$ be as in Theorem \ref{t:FDI}. 
Let $f$ be a positive, continuous and symmetric quasi-convex function with 
$$
0 < \beta \coloneqq \frac{1}{\kappa \log 2} \| f^{-1} \Phi_f \|_{L^\infty} < +\infty.
$$
Suppose that $f$ also satisfies $f^p \in {\rm QC}(\Omega, \mu)$ and $f^{-p} \in L^1(\mu)$ for some $0<p <1/\beta$. 
Then it holds that 
$$
{\rm med}(f) \le \left(\frac{e}{ 1- \beta p}\right)^\beta \|f\|_{L^{-p}(\mu)}.  
$$
\end{proposition}

\begin{proof}
We may suppose that ${\rm med}(f)>0$, otherwise our assertion is obvious. 

We firstly note that we have 
$$
- (1-\theta) \log (1-\theta) \ge \log 2 \min\{\theta, 1-\theta\}, \;\;\; \forall \theta \in [0,1].
$$
Hence, by \eqref{e:DilKappa} and \eqref{e:NegCoareaForm}, we have 
\begin{equation}\label{e:Cal4/15-1}
\kappa \log 2\int_0^{\infty} t^{p-1} \min \{ \mu(\{ x \in \R^n \mid f(x) < t^{-1} \} ),  1-\mu(\{ x \in \R^n \mid f(x) < t^{-1} \} )\} \,dt \le  \int_\Omega f^{-p-1} \Phi_{f} \,d\mu. 
\end{equation}
Since the definition of the L\'{e}vy mean implies that 
$$
\mu(\{ x \in \R^n \mid f(x) < t^{-1} \} ) \le
\mu(\{ x \in \R^n \mid f(x) < {\rm med}(f) \} )
<\frac12, \;\;\; \forall t \ge \frac{1}{{\rm med}(f)}, 
$$
we enjoy 
\begin{align}
&\int_0^{\infty} t^{p-1} \min \{ \mu(\{ x \in \R^n \mid f(x) < t^{-1} \} ),  1-\mu(\{ x \in \R^n \mid f(x) < t^{-1} \} )\} \,dt
\notag \\
\ge&
\int_{\frac{1}{{\rm med}(f)}}^{\infty} t^{p-1} \mu(\{ x \in \R^n \mid f(x) < t^{-1} \} ) \,dt
\notag \\
=&
\int_{0}^{\infty} t^{p-1} \mu(\{ x \in \R^n \mid f(x) < t^{-1} \} ) \,dt
-
\int_0^{\frac{1}{{\rm med}(f)}} t^{p-1} \mu(\{ x \in \R^n \mid f(x) < t^{-1} \} ) \,dt
\notag \\
\ge&
\frac1p \int_\Omega f^{-p}\, d\mu - \frac{1}{p} \left( \frac{1}{{\rm med}(f)} \right)^p. 
\label{e:Cal4/15-2}
\end{align}
On the other hand, it holds that 
\begin{equation}\label{e:Cal4/15-3}
\int_\Omega f^{-p-1} \Phi_{f} \,d\mu
\le
\| f^{-1} \Phi_f \|_{L^\infty} \int_\Omega f^{-p} \,d\mu.
\end{equation}
Since $f^{-p} \in L^1(\mu)$ and $\| f^{-1} \Phi_f \|_{L^\infty} < +\infty$, combining \eqref{e:Cal4/15-1} with \eqref{e:Cal4/15-2} and \eqref{e:Cal4/15-3}, if $\frac{\kappa  \log 2}p > \| f^{-1} \Phi_f \|_{L^\infty}$, we obtain 
$$
\left( \frac{\kappa \log 2}p - \| f^{-1} \Phi_f \|_{L^\infty} \right) \int_\Omega f^{-p} \,d\mu 
\le 
\frac{\kappa \log 2}{p} \left( \frac{1}{{\rm med}(f)} \right)^p.
$$
Therefore it holds that 
$$
{\rm med}(f) \le \left( 1- \frac{p}{\kappa \log 2} \| f^{-1} \Phi_f \|_{L^\infty} \right)^{-\frac{1}{p}} \|f\|_{L^{-p}(\mu)}. 
$$
Since direct calculations yield $(1- t_0 p)^{- \frac1p + t_0} \le e^{t_0}$ for $t_0 >0$ and any $ 0 <p < t_0^{-1}$, we can obtain the desired assertion. 
\end{proof}

Gu\'{e}don \cite{Gu99} (see also \cite[Theorem 2.4.9]{BGVV14}) showed that 
every log-concave probability measure and norm $\|\cdot\|$ on $\R^n$ satisfy 
$$
\int_{\R^n} \|x\|\, d\mu \le \frac{C}{1+q} \left( \int_{\R^n} \|x\|^q\, d\mu \right)^{1/q}, 
$$
for any $-1 < q<0$, where $C>0$ is an absolute constant. 
Hence Proposition \ref{p:NegMom} is a generalization of Gu\'{e}don's result in some sense. 
An extension of Gu\'{e}don's result for general functions is also discussed in \cite{BN07, Fr09}. 
We also remark that Gu\'{e}don's result follows from the small ball estimate, 
$$
\mu\left(\left\{ x \in \R^n \,\Bigg{|}\,  \|x\| \le t \int_{\R^n} \|x\|\, d\mu \right\} \right)
\le Ct, \;\;\; \forall t \ge 1, 
$$
which is shown by Lata\l a \cite{L99}. 
Similarly, we can show the deviation inequality around the origin. 

\begin{corollary}\label{c:NegConc}
Let $\mu$ and $\Omega$ be as in Theorem \ref{t:FDI}. 
Let $f$ be a positive, continuous and symmetric quasi-convex function with 
$$
0 < \beta \coloneqq \frac{1}{\kappa \log 2} \| f^{-1} \Phi_f \|_{L^\infty} < +\infty.
$$
Suppose that $f$ also satisfies $f^p \in {\rm QC}(\Omega, \mu)$ and $f^{-p} \in L^1(\mu)$ for any $0<p <1/\beta$. 
Then for any small enough $\varepsilon>0$, it holds that 
$$
\mu(\{ x \in \Omega \mid f(x) \le t\, {\rm med}(f) \}) \le \left( \frac{e}{\varepsilon \beta} \right)^{1-\varepsilon \beta} t^{\frac1\beta - \varepsilon}, \;\;\; \forall t \in (0, 1].
$$
\end{corollary}

\begin{proof}
Let $p \coloneqq \frac{1}{\beta} -\varepsilon \in (0, \frac1\beta)$ for small enough $\varepsilon>0$. 
It follows from Proposition \ref{p:NegMom} that 
$$
\int_\Omega f^{-p}\, d\mu \le \left( \frac{e^\beta}{{\rm med}(f) (1-\beta p)^\beta} \right)^p.
$$
Hence Markov's inequality implies that 
$$
\mu(\{ x \in \Omega \mid f(x) \le t\, {\rm med}(f) \})
\le
\left( \frac{e}{1-\beta p} \right)^{\beta p} t^p, \;\;\; \forall t \in (0, 1].
$$
This implies the desired assertion by $p=\frac1\beta -\varepsilon$. 
\end{proof}

\begin{proof}[Proof of Corollary \ref{c:LDIntro} (2)]
Let $f$ be a positive $C^1$ symmetric quasi-convex function on some neighborhood of $\overline{\Omega}$ with $0<\beta < +\infty$ and $f^{-1/\beta} \in L^1(\mu)$. 
Then for any $p>0$, $f^p$ is also a positive $C^1$ symmetric quasi-convex function on some neighborhood of $\overline{\Omega}$, and thus $f^p \in {\rm QC}(\mu, \Omega)$ by Proposition \ref{p:ClassQC} since $\Omega$ is bounded. 
Moreover, for any $0<p<1/\beta$, we have $f^{-p} \in L^1(\mu)$ by $f^{-1/\beta} \in L^1(\mu)$ and H\"{o}lder's inequality. Hence we see that $f$ satisfies the assumptions in Corollary \ref{c:NegConc}. 
Applying Corollary \ref{c:NegConc} to $f$, we can conclude the desired assertion. 
\end{proof}

%
%

\subsection{$\mu$-perimeter}

Our goal in this subsection is to give the estimate of the $\mu$-perimeter of $K \in \K_s^n(\Omega)$ described in Corollary \ref{c:IsoIntro}. 

\begin{proof}[Proof of Corollary \ref{c:IsoIntro}]
Without loss of generality, we may suppose that $\mu(\overline{K})=\mu(K)$. 
We set for $\varepsilon>0$, 
$$
f_{\varepsilon}(x) \coloneqq \min\left\{1, \frac{1}{\varepsilon} d(x, K) \right\}, \;\;\; x \in \R^n. 
$$
Then $f_{\varepsilon}$ is a locally Lipschitz function on $\{x \in \R^n \mid f_{\varepsilon}(x) >0\}$. 
Moreover, we can check that $f_{\varepsilon}$ is a nonnegative and symmetric quasi-convex function with $f_\varepsilon \in L^1(\mu)$.
Hence it follows from Theorem \ref{t:FDIIntro} that 
\begin{equation}\label{e:4/24Cal2}
\Ent_\mu(f_{\varepsilon}) 
\le
\frac{2}{\kappa}  \int_{\Omega} |x||\nabla f_{\varepsilon}(x)|\, d\mu(x). 
\end{equation}
Since we see  that $\nabla f_{\varepsilon}(x) =0$ if $x \in K \cup (\R^n \setminus \overline{[K]_\varepsilon})$ and $|\nabla f_{\varepsilon}| \le \varepsilon^{-1}$ on $\R^n$, it holds that 
\begin{align*}
 &\int_{\Omega} |x||\nabla f_{\varepsilon}(x)|\, d\mu(x)
 \notag \\
 \le& 
 \frac{1}{\varepsilon} \int_\Omega |x| \mathbf{1}_{[K]_{\varepsilon} \setminus K}\, d\mu(x)
 \notag \\
 \le&
 \left( \frac{1}{\varepsilon} \int_\Omega |x|^{p'} \mathbf{1}_{[K]_{\varepsilon} \setminus K}(x)\, d\mu(x) \right)^{1/p'}
 \left( \frac{1}{\varepsilon}  \int_\Omega\mathbf{1}_{[K]_{\varepsilon} \setminus K}(x)\, d\mu(x) \right)^{1/p}
 \notag \\
 =&
 \left( \frac{1}{\varepsilon} \int_\Omega |x|^{p'} \mathbf{1}_{[K]_{\varepsilon} \setminus K}(x)\, d\mu(x) \right)^{1/p'}
 \left( \frac1\varepsilon(\mu([K]_\varepsilon) - \mu(K) )\right)^{1/p},
\end{align*}
where we used H\"{o}lder's inequality. 
Furthermore, since we have by $r(K) {\rm B}_2^n \subset K$, 
$$
[K]_\varepsilon = K + \varepsilon {\rm B}_2^n \subset \left(1 + \frac{\varepsilon}{r(K)} \right)K, 
$$
it holds that 
\begin{align*}
&
\frac{1}{\varepsilon} \int_\Omega |x|^{p'} \mathbf{1}_{[K]_{\varepsilon} \setminus K}(x)\, d\mu(x)
\\
= &
 \frac{1}{\varepsilon} \left(\int_{[K]_{\varepsilon}} |x|^{p'} \, d\mu(x) - \int_K |x|^{p'} \, d\mu(x) \right)
\\
\le&
\frac{1}{\varepsilon} \left(\int_{(1 + \frac{\varepsilon}{r(K)})K} |x|^{p'} \, d\mu(x) - \int_K |x|^{p'} \, d\mu(x) \right)
\\
=&
\int_K \frac1\varepsilon \left[ \left(1 + \frac{\varepsilon}{r(K)} \right)^{p'+n} e^{-\varphi((1 + \frac{\varepsilon}{r(K)})x)} - e^{-\varphi(x)} \right] |x|^{p'}\, dx. 
\end{align*}
Since $K$ is bounded and $\varphi$ is smooth, Fatou's lemma yields that 
\begin{align*}
&
\liminf_{\varepsilon \downarrow 0} \frac{1}{\varepsilon} \int_\Omega |x|^{p'} \mathbf{1}_{[K]_{\varepsilon} \setminus K}(x)\, d\mu(x)
\le
\frac{1}{r(K)} \int_K (p' + n - \langle x, \nabla \varphi(x) \rangle) |x|^{p'} e^{-\varphi(x)}\, dx. 
\end{align*}
Moreover, since we enjoy 
$$
{\rm div}(x |x|^{p'}e^{-\varphi(x)})
=
(p' + n - \langle x, \nabla \varphi(x) \rangle) |x|^{p'} e^{-\varphi(x)}, 
$$
the divergence theorem implies that 
\begin{equation}\label{e:4/24Cal}
\liminf_{\varepsilon \downarrow 0} \frac{1}{\varepsilon} \int_\Omega |x|^{p'} \mathbf{1}_{[K]_{\varepsilon} \setminus K}(x)\, d\mu(x)
\le
\frac{1}{r(K)} \int_{\partial K} \langle x, \eta(x) \rangle |x|^{p'}e^{-\varphi(x)}\, d\sigma_K(x). 
\end{equation}
Therefore, letting $\varepsilon \downarrow 0$ in \eqref{e:4/24Cal2}, we obtain 
$$
\liminf_{\varepsilon \downarrow 0} \Ent_{\mu}(f_\varepsilon)
\le
\frac{2}{\kappa} \left( \frac{1}{r(K)} \int_{\partial K} \langle x, \eta(x) \rangle |x|^{p'}e^{-\varphi(x)}\, d\sigma_K(x) \right)^{1/p'} 
\mu^+(K)^{1/p}.
$$
Since $\lim_{\varepsilon \downarrow 0} f_\varepsilon = \mathbf{1}_{\R^n \setminus \overline{K}}$ and $\mu(\overline{K})=\mu(K)$, the lower semi-continuity of the relative entropy yields that 
$$
- (1-\mu(K) ) \log (1-\mu(K))
\le
\frac{2}{\kappa} \left( \frac{1}{r(K)} \int_{\partial K} \langle x, \eta(x) \rangle |x|^{p'}e^{-\varphi(x)}\, d\sigma_K(x) \right)^{1/p'} 
\mu^+(K)^{1/p}, 
$$
which implies the desired assertion.

\end{proof}

\section{Revisit to the dilation inequality}\label{s:Re}

\subsection{Reconstruction}

%

In Section \ref{s:FDI}, we investigated the functional form of the dilation inequality. 
In this section, conversely, we will confirm that the dilation inequality can be recover from the functional inequality \eqref{e:FDI}.

For $K \in \K_s^n(\R^n)$, we define a function $\mathcal{N}_K \colon \R^n \to [0, \infty)$ by 
$$
\mathcal{N}_K(x) \coloneqq 
\begin{cases}
\|x\|_K & \text{if $x \notin K$} \\
1 & \text{if $x \in K$}. 
\end{cases}
$$ 
Then we can easily check that $\mathcal{N}_K$ is a continuous, symmetric and quasi-convex function on $\R^n$.

\begin{theorem}\label{t:ReConstDI}
Let $\mu$ be a probability measure supported on a symmetric convex domain $\Omega \subset \R^n$ with $\int_\Omega |x| \, d\mu(x) < + \infty$. 
We suppose that \eqref{e:FDI} holds for any $f \in \mathrm{QC}^1(\Omega, \mu)$ with some $\kappa>0$. 
Then $\mu$ satisfies the dilation inequality \eqref{e:DilKappa} for $\K_s^n(\Omega)$ with the constant $\kappa$. 
\end{theorem}


\begin{proof}
Let us fix $K \in \K_s^n(\Omega)$. 
We first remark that we may assume $\mu(\overline{K}) = \mu(K)$, otherwise we have $\mu^*(K) = + \infty$ since 
$$
\frac{\mu(K_\varepsilon) - \mu(K)}{\varepsilon} 
\ge
\frac{\mu(\overline{K}) - \mu(K)}{\varepsilon} 
$$
and thus nothing to prove. 

Let $\sigma \in (0, 1)$ and set $\delta \coloneqq 2\sigma/(1-\sigma)^2$.
We define 
\begin{align*}
f_\sigma(x) 
\coloneqq& 
\min \left\{ \frac{1}{\delta}(\mathcal{N}_K(x) -1), 1 - \sigma \right\}
\\
=&
\begin{cases}
\frac{1}{\delta}(\mathcal{N}_K(x) -1) & \;\;\; \text{if $x \in K_{\sigma}$} \\
1 - \sigma & \;\;\; \text{if $x \notin K_{\sigma} $}
\end{cases}
\\
=&
\begin{cases}
0 & \text{if $x \in K$} \\
\frac{1}{\delta}(\|x\|_K -1) & \text{if $x \in K_\sigma \setminus K$} \\
1-\sigma & \text{if $x \notin K_\sigma$}
\end{cases}
\end{align*}
for $x \in \R^n$.
Note that $f_\sigma$ is a nonnegative, symmetric, continuous and quasi-convex  function. 
We can also check that $f_\sigma \in L^1(\mu)$ since $f_\sigma \le 1-\sigma$ on $\R^n$. 
Moreover, we can obtain $f_\sigma \in {\rm QC}^1(\Omega, \mu)$. 
To see this, we shall justify \eqref{e:CondQC} for $f_\sigma$. For any $\varepsilon \in (0, 1)$ and $x \in \R^n$, it holds 
\begin{align*}
f_\sigma \left(\frac{1-\varepsilon}{1+\varepsilon}x \right)
=
\begin{cases}
0 & \text{if $x \in K_\varepsilon$} \\
\frac{1}{\delta}(\frac{1-\varepsilon}{1+\varepsilon}\|x\|_K -1) & \text{if $x \in (K_\sigma)_\varepsilon \setminus K_\varepsilon$} \\
1-\sigma & \text{if $x \notin (K_\sigma)_\varepsilon$}, 
\end{cases}
\end{align*}
where we used $\frac{1+\varepsilon}{1-\varepsilon}K = K_\varepsilon$. 
Hence for any $\varepsilon \in (0, \sigma)$, we have 
\begin{align*}
\frac{1}{\varepsilon} \left( f_\sigma(x) - f_\sigma \left(\frac{1-\varepsilon}{1+\varepsilon}x \right) \right)
=
\begin{cases}
\frac{1}{\varepsilon} \frac{1}{\delta}(\|x\|_K -1) & \text{if $x \in K_\varepsilon \setminus K$} \\
\frac{1}{\varepsilon} ( \frac{1}{\delta}(\|x\|_K -1) - \frac{1}{\delta}(\frac{1-\varepsilon}{1+\varepsilon}\|x\|_K -1) ) & \text{if $x \in K_\sigma \setminus K_\varepsilon$} \\
\frac{1}{\varepsilon}  ( 1-\sigma - \frac{1}{\delta}(\frac{1-\varepsilon}{1+\varepsilon}\|x\|_K -1) ) & \text{if $x \in (K_\sigma)_\varepsilon \setminus K_\sigma$} \\
0 & \text{otherwise}. 
\end{cases}
\end{align*}
For any $x \in K_\varepsilon$, since $\frac{1-\varepsilon}{1+\varepsilon}\|x\|_K = \|x\|_{K_\varepsilon} \le 1$, we have 
$$
\frac{1}{\varepsilon} \frac{1}{\delta}(\|x\|_K -1)
\le
\frac{1}{\varepsilon} \frac{1}{\delta} \left(\|x\|_K - \frac{1-\varepsilon}{1+\varepsilon}\|x\|_K\right)
=
\frac{2}{\delta(1+\varepsilon)} \|x\|_K
\le
\frac{2}{\delta} \|x\|_K. 
$$
Next for $x \not \in K_\sigma$, we have 
$$
1-\sigma \le \frac{1}{\delta}(\|x\|_K-1)
$$
by $\|x\|_K \ge \frac{1+\sigma}{1-\sigma}$. 
Thus it holds that 
\begin{align*}
\frac{1}{\varepsilon}  \left( 1-\sigma - \frac{1}{\delta} \left(\frac{1-\varepsilon}{1+\varepsilon}\|x\|_K -1 \right) \right)
\le&
\frac{1}{\varepsilon}  \left( \frac{1}{\delta}(\|x\|_K-1) - \frac{1}{\delta} \left(\frac{1-\varepsilon}{1+\varepsilon}\|x\|_K -1 \right) \right)
\\
=&
\frac{2}{\delta(1+\varepsilon)} \|x\|_K
\\
\le&
\frac{2}{\delta} \|x\|_K.
\end{align*}
In particular, for any $x \in \R^n$, we have 
\begin{align*}
\frac{1}{\varepsilon} \left( \frac{1}{\delta} (\|x\|_K -1 ) - \frac{1}{\delta} \left(\frac{1-\varepsilon}{1+\varepsilon}\|x\|_K -1 \right) \right)
\le
\frac{2}{\delta} \|x\|_K.
\end{align*}
Therefore summarizing our arguments above, we can conclude that 
$$
\frac{1}{\varepsilon} \left( f_\sigma(x) - f_\sigma \left(\frac{1-\varepsilon}{1+\varepsilon}x \right) \right)
\le
\frac{2}{\delta} \|x\|_K
$$
for any $x \in \R^n$ and $\varepsilon \in (0, \sigma)$, which ensures \eqref{e:CondQC} for $f_\sigma$ since we can take some constant $C>0$ such that $\| \cdot \|_K \le C|\cdot|$ and since we have $\int_\Omega |x| \, d\mu<+\infty$.
Hence we could check $f_\sigma \in {\rm QC}^1(\Omega, \mu)$. 
Moreover, we can enjoy 
$$
\Phi_{f_\sigma}(x)
=
\frac{2}{\delta} \|x\|_K \mathbf{1}_{\overline{K_\sigma} \setminus \overline{K}}(x), \;\;\; x \in \R^n.
$$
Hence we can apply \eqref{e:FDI} to $f_\sigma$ to see 
$$
\Ent_{\mu}(f_\sigma)
\le
\frac{1}{\kappa} \int_\Omega \Phi_{f_\sigma}\, d\mu.
$$
In the right hand side, since it holds that 
$$
\|x\|_K \le \frac{1+\sigma}{1-\sigma}, \;\;\; x \in \overline{K_\sigma}, 
$$
we obtain 
\begin{align*}
\int_\Omega \Phi_{f_\sigma}\, d\mu
\le&
\frac{2}{\delta} \frac{1+\sigma}{1-\sigma} (\mu(\overline{K_\sigma}) - \mu(\overline{K}) )
\le
(1-\sigma^2) \frac{1}{\sigma} (\mu(K_{(1 + \tau) \sigma}) - \mu(K) ) 
\end{align*}
for any small enough $\tau>0$, where we used $\overline{K_\sigma} \subset K_{(1+\tau)\sigma}$ and $\mu(\overline{K})=\mu(K)$ in the last inequality. 
Hence we have 
$$
\liminf_{\sigma \downarrow 0} \int_\Omega \Phi_{f_\sigma}\, d\mu
\le (1+\tau)\mu^*(K)
$$
for any small enough $\tau>0$, and thus 
$$
\liminf_{\sigma \downarrow 0} \int_\Omega \Phi_{f_\sigma}\, d\mu
\le \mu^*(K). 
$$
On the other hand, since it holds that $f_\sigma \to \mathbf{1}_{\R^n \setminus \overline{K}}$ as $\sigma \downarrow 0$, 
it follows from the lower semi-continuity of $\Ent_\mu$ that 
$$
\liminf_{\sigma \downarrow 0} \Ent_\mu (f_\sigma)
\ge 
\Ent_\mu (\mathbf{1}_{\R^n \setminus \overline{K}})
=
- (1-\mu(\overline{K}) ) \log (1-\mu(\overline{K}))
=
- (1-\mu(K) ) \log (1-\mu(K))
$$
by $\mu(\overline{K})=\mu(K)$. 
Eventually, we have 
$$
- (1-\mu(K) ) \log (1-\mu(K))
\le
\frac1\kappa \mu^*(K), 
$$
which is the desired assertion. 
\end{proof}

\subsection{Applications}

As a corollary of Theorem \ref{t:ReConstDI}, 
we obtain the stability of the dilation inequality for bounded perturbations, which is described in Corollary \ref{c:BDDIntro}. 
To show this corollary, we employ the following lemma. 


\begin{lemma}[{\cite[Lemma 5.1.7]{BGL14}}]\label{l:BGLp240}
For any nonnegative function $f \in L^1(\R^n, \mu)$ with $\int_{\R^n} f\, d\mu>0$, it holds 
$$
\Ent_\mu(f)
=
\inf_{r \in (0, \infty)} \int_{\R^n} [ \phi(f) - \phi(r) - \phi'(r) (f-r)]\, d\mu. 
$$
Here we set $\phi(r) \coloneqq r\log r$ for $r>0$.
\end{lemma}

\begin{proof}[Proof of Corollary \ref{c:BDDIntro}]
We firstly note that since $h$ is bounded from above and below by a positive constant, we have ${\rm QC}^1(\Omega, \mu) = {\rm QC}^1(\Omega, \nu)$. 
Moreover it follows from $h \le b$ and Lemma \ref{l:BGLp240} that 
$
\Ent_{\nu}(f) \le b \Ent_\mu(f)
$ 
for any $f \in {\rm QC}^1(\Omega, \nu)$. 
Since $\mu$ satisfies the dilation inequality for $\K_s^n(\Omega)$ with $\kappa>0$, we can apply Theorem \ref{t:FDI} to see 
$$
\Ent_{\nu}(f) \le \frac{b}{\kappa} \int_{\Omega} \Phi_f\, d\mu
$$
for any $f \in {\rm QC}^1(\Omega, \nu)$. 
By $h \ge b^{-1}$, we conclude 
$$
\Ent_{\nu}(f) \le \frac{b^2}{\kappa} \int_{\Omega} \Phi_f\, d\nu
$$
for any $f \in {\rm QC}^1(\Omega, \nu)$. 


Now we can check $\int_\Omega |x|\, d\nu < +\infty$ since $\int_\Omega |x|\, d\mu < +\infty$ and $h$ is bounded from above. 
Hence we obtain the desired assertion by applying Theorem \ref{t:ReConstDI} to $\nu$. 
\end{proof}

As another consequence of Theorem \ref{t:ReConstDI}, we can observe the tensorization property in the special case, which is described in Corollary \ref{c:TensorProdIntro}. 
We remark that $\Omega_1 \times \Omega_2$ is also a symmetric convex domain in $\R^{n+1}$. 


\begin{proof}[Proof of Corollary \ref{c:TensorProdIntro}]
Let ${\rm QC}^{1,*}(\Omega_1 \times \Omega_2, \mu_1 \otimes \mu_2)$ be the set of all functions $f \in {\rm QC}^1(\Omega_1 \times \Omega_2, \mu_1 \otimes \mu_2)$ such that $f$ is bounded on $\Omega_1 \times \Omega_2$ and satisfies 
\begin{equation}\label{e:Uncond}
f(x, y) = f(-x, y) = f(x, -y) = f(-x,-y), \;\;\; \forall (x, y) \in \Omega_1 \times \Omega_2.
\end{equation}
Now, fix $f \in {\rm QC}^{1,*}(\Omega_1 \times \Omega_2, \mu_1 \otimes \mu_2)$, and set 
$$
g(x) \coloneqq \int_{\Omega_2} f(x, y)\, d\mu_2(y). 
$$
Then we see that 
\begin{align*}
&\Ent_{\mu_1 \otimes \mu_2}(f)
\\
=& \Ent_{\mu_1}(g) 
+
\int_{\Omega_1} \left(
\int_{\Omega_2} f(x, y) \log f(x, y) \, d\mu_2(y) 
-
\int_{\Omega_2} f(x, y)\, d\mu_2(y) \log \int_{\Omega_2} f(x, y)\, d\mu_2(y)
\right)
d\mu_1(x).
\end{align*}
If we have $g \in {\rm QC}^1(\Omega_1, \mu_1)$ and $f(x, \cdot) \in {\rm QC}^1(\Omega_2, \mu_2)$ for $\mu_1$-a.e., $x \in \Omega_1$, then we can apply Theorem \ref{t:FDI} to see that 
\begin{equation}\label{e:Cal4/6}
\Ent_{\mu_1 \otimes \mu_2}(f)
\le \frac{1}{\kappa_1} \int_{\Omega_1} \Phi_g(x)\, d\mu_1(x)
+
\frac{1}{\kappa_2} \int_{\Omega_1} \int_{\Omega_2} \Phi_{f(x, \cdot)}(y)\, d\mu_2(y) d\mu_1(x).
\end{equation}
By the definition, we have 
\begin{align*}
\Phi_g(x)
=
\limsup_{\varepsilon \downarrow 0}
\int_{\Omega_2} \frac{1}{\varepsilon} \left( f(x, y) - f \left(\frac{1-\varepsilon}{1+\varepsilon}x, y \right) \right)\, d\mu_2(y). 
\end{align*}
Now we remark that $f (\frac{1-\varepsilon}{1+\varepsilon}x, y ) \ge f (\frac{1-\varepsilon}{1+\varepsilon}x, \frac{1-\varepsilon}{1+\varepsilon}y)$ for any $(x, y) \in \Omega_1 \times \Omega_2$. Indeed, since $f$ is symmetric quasi-convex function on $\Omega_1 \times \Omega_2$ and satisfies \eqref{e:Uncond}, $f(z, \cdot)$ is symmetric quasi-convex on $\Omega_2$ for each $z \in \Omega_1$. 
In particular, $f(z, ty)$ is monotone increasing in $t\ge0$ for any $y \in \Omega_2$, and hence we conclude $f (\frac{1-\varepsilon}{1+\varepsilon}x, y ) \ge f (\frac{1-\varepsilon}{1+\varepsilon}x, \frac{1-\varepsilon}{1+\varepsilon}y)$. 
From this and $f \in {\rm QC}^1(\Omega_1 \times \Omega_2, \mu_1 \otimes \mu_2)$, we obtain 
$$
\frac{1}{\varepsilon} \left( f(x, y) - f \left(\frac{1-\varepsilon}{1+\varepsilon}x, y \right) \right)
\le 
\frac{1}{\varepsilon} \left( f(x, y) - f \left(\frac{1-\varepsilon}{1+\varepsilon}x, \frac{1-\varepsilon}{1+\varepsilon}y \right) \right)
\le
h(x, y)
$$
for all small enough $\varepsilon >0$ and $(x, y) \in \Omega_1 \times \Omega_2$, where $h$ is a Borel function in $L^1(\Omega_1 \times \Omega_2, \mu_1 \otimes \mu_2)$. 
In particular, $h(x, \cdot) \in L^1(\Omega_2, \mu_2)$ holds for $\mu_1$-a.e., $x \in \Omega_1$. 
Hence it follows from Fatou's lemma that 
$$
\Phi_g(x)
\le
\int_{\Omega_2} \Phi_f(x, y)\, d\mu_2(y)
$$
for $\mu_1$-a.e., $x \in \Omega_1$. 
Similarly, we can observe that 
$ \Phi_{f(x, \cdot)}(y) \le \Phi_f(x, y)$. 
Therefore we can conclude 
\begin{equation}\label{e:Ent5/5}
\Ent_{\mu_1 \otimes \mu_2}(f)
\le
\left(\frac{1}{\kappa_1} + \frac{1}{\kappa_2} \right) \int_{\Omega_1 \times \Omega_2} \Phi_f(x, y) \, d\mu_1 \otimes \mu_2(x, y)
\end{equation}
for any $f \in {\rm QC}^{1,*}(\Omega_1 \times \Omega_2, \mu_1 \otimes \mu_2)$. 

Now let $K \subset \R \times \R^n$ be a symmetric open convex set such that if $(x, y) \in \R \times \R^n$ belongs to $K$, then $(-x, y), (x, -y), (-x, -y)$ also belong to $K$. 
Let us consider the function $f_\sigma$ given in the proof of Theorem \ref{t:ReConstDI} for $\sigma \in (0, 1)$ and $K$. 
Then by the property of $K$ and the definition of $f_\sigma$, we see that $f_\sigma \in {\rm QC}^{1,*}(\Omega_1 \times \Omega_2, \mu_1 \otimes \mu_2)$. We also remark that it holds that $\int_{\R \times \R^n} |(x, y)|\, d\mu_1\otimes\mu_2(x, y) < +\infty$ by assumptions. 
Hence, by iterating the same arguments for $f_\sigma$ as in the proof of Theorem \ref{t:ReConstDI} via \eqref{e:Ent5/5}, we can obtain 
$$
\mu^*(K) \ge - \left(\frac{1}{\kappa_1} + \frac{1}{\kappa_2} \right)^{-1} (1-\mu(K)) \log (1-\mu(K)), 
$$
which is the desired assertion. 

Finally, to justify the above argument, we prove $g \in {\rm QC}^1(\Omega_1, \mu_1)$ and $f(x, \cdot) \in {\rm QC}^1(\Omega_2, \mu_2)$ for $\mu_1$-a.e., $x \in \Omega_1$. 
Since $f$ is bounded and satisfies \eqref{e:Uncond}, we can easily check that $f(x, \cdot)$ and $g$ are nonnegative, continuous and symmetric functions, and $f(x, \cdot)$ is quasi-convex. 
In addition, since $f$ is integrable, we see that $g \in L^1(\Omega_1, \mu_1)$ and $f(x, \cdot) \in L^1(\Omega_2, \mu_2)$ for $\mu_1$-a.e., $x \in \Omega_1$. 
Since we have already shown that $g$ and $f(x, \cdot)$ satisfy the condition \eqref{e:CondQC} by the above argument, we obtain $f(x, \cdot) \in {\rm QC}^1(\Omega_2, \mu_2)$ for $\mu_1$-a.e., $x \in \Omega_1$. 
Therefore it suffices to show that $g$ is quasi-convex from which we can conclude $g \in {\rm QC}^1(\Omega_1, \mu_1)$. 

Let $\lambda>0$ and consider $A(\lambda) \coloneqq \{ x \in \R \mid g(x) < \lambda\}$. 
Without loss of generality, we may assume that $A(\lambda) \neq \emptyset$. 
We remark that $A(\lambda)$ is symmetric since $g$ is symmetric. 
Now take $z \in A(\lambda)$ with $z\ge0$. 
Since $f$ is quasi-convex, $f(\cdot, y)$ is also quasi-convex for any $y \in \Omega_2$. 
In particular, since $f(\cdot, y)$ is symmetric by \eqref{e:Uncond}, we see that $f(tz, y)$ is monotone increasing in $t \ge 0$. 
Thus for any $z' \in [0, z]$ and $y \in \Omega_2$, we have $f(z', y) \le f(z, y)$ which implies that $g(z') \le g(z) <\lambda$. 
Since $g$ is symmetric, we obtain $[-z, z] \subset A(\lambda)$ from which $A(\lambda)$ should be an interval. 
Hence $g$ is quasi-convex, and our proof is complete. 
\end{proof}

\begin{remark}
It is natural to expect the tensorization property for high dimensional spaces. 
More precisely, if $\mu_1$ and $\mu_2$ are probability measures on $\R^{n_1}$ and $\R^{n_2}$ satisfying dilation inequalities for $\K_s^{n_1}(\R^{n_1})$ and $\K_s^{n_2}(\R^{n_2})$, respectively, then does $\mu_1 \otimes \mu_2$ also satisfy the dilation inequality for $\K_s^{n_1 + n_2}(\R^{n_1+n_2})$? 
Corollary \ref{c:TensorProdIntro} gives a partial answer affirmatively when either $n_1$ or $n_2$ is 1, but it is open when $n_1, n_2 \ge 2$. 
In our argument, this difficulty comes from quasi-convexity. 
In fact, let $f_1$ and $f_2$ be nonnegative  symmetric quasi-convex functions on $\R^n$. 
Then in our proof of Corollary \ref{c:TensorProdIntro}, we used the fact that $f_1 + f_2$ is also a symmetric quasi-convex function when $n=1$. 
However, when $n\ge 2$, the same phenomenon fails. 
For instance, let us consider functions $f_1(x_1, x_2) = |x_1|^{2/3}$ and $f_2(x_1, x_2) = |x_2|^{2/3}$ for $(x_1, x_2) \in \R^2$. Then we can check that both functions are symmetric and quasi-convex, but $f_1 + f_2$ is not quasi-convex on $\R^2$ (the curve $\{ (x_1, x_2) \in \R^2 \mid f_1(x_1, x_2) + f_2(x_1, x_2)=1\}$ is the astroid). 
\end{remark}

In our proof of Corollary \ref{c:TensorProdIntro}, we also showed the tensorization property of the functional dilation inequality \eqref{e:FDI}. 
If we focus on this tensorization, we can improve the functional version of Corollary  \ref{c:TensorProdIntro} in the special case. 

\begin{corollary}\label{c:TensorProdF}
Let $\mu_1, \mu_2$ be probability measures supported on symmetric convex domains $\Omega_1 \subset \R$ and $\Omega_2 \subset \R^n$, respectively. 
We suppose that $\mu_1, \mu_2$ satisfy \eqref{e:FDI} with some $\kappa_1, \kappa_2>0$, respectively. 
Then $\mu_1 \otimes \mu_2$ satisfies \eqref{e:FDI} with the constant $\kappa=\min\{\kappa_1, \kappa_2\}$ for any bounded function $f \in {\rm QC}^1(\Omega_1 \times \Omega_2, \mu_1\otimes \mu_2) \cap C^1(\Omega_1 \times \Omega_2)$ satisfying  \eqref{e:Uncond}. 
\end{corollary}

\begin{proof}
Almost arguments are same in the proof of Corollary \ref{c:TensorProdIntro}, but we remark that since $f \in C^1(\Omega_1 \times \Omega_2)$, we have 
$$
\Phi_g(x) = 2\int_{\Omega_2} \langle x, \nabla_x f(x, y) \rangle\, d\mu_2(y), \;\;\; x \in \Omega_1
$$
and
$$
\Phi_{f(x, \cdot)}(y) = 2\langle y, \nabla_y f(x, y) \rangle, \;\;\;
(x, y) \in \Omega_1 \times \Omega_2.
$$
Hence since we also see that 
$$
\Phi_f(x, y) = 2 \langle (x, y), \nabla f(x,y) \rangle, \;\;\; (x, y) \in \Omega_1 \times \Omega_2, 
$$ 
it follows from \eqref{e:Cal4/6} that  
\begin{align*}
&\Ent_{\mu_1 \otimes \mu_2}(f)
\\
\le& \max\left\{ \frac{1}{\kappa_1}, \frac{1}{\kappa_2} \right\} 
\left( 2\int_{\Omega_1} \int_{\Omega_2} \langle x, \nabla_x f(x, y) \rangle\, d\mu_2(y)\, d\mu_1(x)
+
 2\int_{\Omega_1} \int_{\Omega_2} \langle y, \nabla_y f(x, y) \rangle\, d\mu_2(y) d\mu_1(x)
 \right)
 \\
 =&
\frac{2}{\min\{\kappa_1, \kappa_2\}}
\int_{\Omega_1 \times \Omega_2} \langle (x,y), \nabla f(x, y) \rangle\, d\mu_1 \otimes \mu_2(x,y)
\\
 =&
\frac{1}{\min\{\kappa_1, \kappa_2\}}
\int_{\Omega_1 \times \Omega_2} \Phi_f(x, y) \, d\mu_1 \otimes \mu_2(x,y), 
\end{align*}
which is the desired assertion.
\end{proof}

\appendix

\section{Appendix}

Here we will investigate the dilation inequality for symmetric log-concave probability measures on $\R$ and the standard Gaussian measure $\gamma_n$ on $\R^n$. 

\begin{proposition}\label{p:LC}
Every symmetric log-concave probability measure on $\R$ satisfies the dilation inequality for $\K_s^1(\R)$ with $\kappa=2$. 
\end{proposition}

\begin{proof}
Let $\mu=e^{-\varphi(x)}\, dx$ be a symmetric probability measure on $\R$. 
Then we see that 
$\mu^*((-t, t)) = 4 e^{-\varphi(t)}t$ for any $t>0$ (see \cite{Ts21}). 
We set $d\nu \coloneqq 2e^{-\varphi(x)} \mathbf{1}_{(0, \infty)}\, dx$, then $\nu$ is a log-concave probability measure since $\mu$ is symmetric. 
Moreover, we enjoy $\nu^*((0, t)) = 2e^{-\varphi(t)}t$ for any $t>0$. 
Hence, we obtain 
$\mu^*((-t, t)) = 2 \nu^*((0, t))$. 
On the other hand, since every log-concave probability measure satisfies the dilation inequality \eqref{e:Borell}, we can conclude 
$$
\mu^*((-t, t)) \ge - 2 (1-\nu((0, t))) \log (1-\nu((0,t))). 
$$
Finally, since we have $\nu((0, t)) = \mu((-t, t))$ by symmetry of $\mu$, we obtain the desired assertion. 
\end{proof}

\begin{proposition}\label{p:GaussDil}
The standard Gaussian measure $\gamma_n$ satisfies the dilation inequality for $\K_s^n(\R^n)$ with $\kappa=2$. 
\end{proposition}

\begin{proof}
We employ the result by Lata\l a--Oleszkiewicz \cite{LO99} where they showed that 
for any $K \in \K_s^n(\R^n)$, it holds that 
$$
\gamma_n(tK) \ge \gamma_n(tP), \;\;\; \forall t \ge 1, 
$$
where $P \subset \R^n$ is a symmetric strip such that $\gamma_n(P)=\gamma_n(K)$. 
In particular, we can take $\theta_K>0$ satisfying $\gamma_n(P) = \gamma_n(\R^{n-1} \times (-\theta_K, \theta_K))$, $\gamma_n(tK) \ge \gamma_1((-t\theta_K, t\theta_K))$ for any $t \ge 1$ and $\gamma_n(K) = \gamma_1((-\theta_K, \theta_K))$.
Hence, by the definition of the dilation area, we can conclude 
$\gamma_n^*(K) \ge \gamma_1^*((-\theta_K, \theta_K))$. 
Finally, since $\gamma_1$ is log-concave, Proposition \ref{p:LC} implies the desired assertion. 
\end{proof}

We also remark that $\kappa=2$ is optimal in Proposition \ref{p:GaussDil}. 
Indeed, we can check that for $t>0$, 
$$
\gamma_1^*((-t, t)) 
=
\frac{4}{\sqrt{2\pi}} e^{-\frac12 t^2} t
$$
and 
$$
- (1-\gamma_1((-t,t))) \log (1-\gamma_1((-t,t)))
=
\frac{2}{\sqrt{2\pi}} t + o(t)
$$
as $t \to +0$. 
Hence if $\gamma_1$ satisfies \eqref{e:DilKappa} for $\K_s^1(\R)$ with $\kappa>0$, then $\kappa$ should satisfy $\kappa \le 2$.

\section*{Acknowledgments}
The author would like to thank Professors Shin-ichi Ohta and Shohei Nakamura for helpful comments. 
This work was supported partially by JST, ACT-X Grant Number JPMJAX200J, Japan,
and JSPS Kakenhi grant number 22J10002.

\end{document}